\documentclass[11pt,reqno]{amsart}
\usepackage{amscd,amssymb,amsmath,amsthm}
\usepackage{mathtools}
\usepackage[english]{babel}
\usepackage[T1]{fontenc}
\usepackage[arrow,matrix]{xy}
\usepackage{graphicx,tikz,tikz-3dplot}
\usepackage{caption}
\usepackage{subcaption}

\usetikzlibrary{positioning,matrix,arrows,calc}
\usepackage{dsfont}
\usepackage{enumitem}
\usepackage{bbm}
\usepackage{hyperref}

\hypersetup{
	colorlinks=true,
	linkcolor=blue}
\usepackage{array}
\usepackage{xargs}                      
\usepackage[colorinlistoftodos,prependcaption,textsize=tiny]{todonotes}

\usepackage[
   backend=biber,        
   sorting=nyt,          
   citestyle=alphabetic, 
   bibstyle=alphabetic,  
   isbn=false,
   doi=false,
   url=false,
   maxalphanames=99,
   maxnames=99,
   giveninits=true,
   useprefix=true, 
]{biblatex}
\theoremstyle{plain}
\newtheorem{theorem}{Theorem}[section]
\newtheorem{lemma}[theorem]{Lemma}
\newtheorem{proposition}[theorem]{Proposition}

\theoremstyle{definition}
\newtheorem{remark}[theorem]{Remark}
\newtheorem{definition}[theorem]{Definition}

\numberwithin{equation}{section} 

\newcommand{\N}{\mathbb{N}}
\newcommand{\R}{\mathbb{R}}

\newcounter{hypocounter}
\renewcommand\thehypocounter{(H\arabic{hypocounter})} 
\newcommand{\Ent}{\operatorname{Ent}}
\renewcommand{\Cap}{\operatorname{Cap}_\alpha}

\newcommand{\Gibbs}{\mu^{\eta}_{N,\alpha, \lambda, \beta}} 
\newcommand{\Gib}{\mu^{\eta}_{N}} 
\newcommand{\EE}{E^\eta_{N,\alpha, \lambda, \beta}} 
\newcommand{\E}{E^\eta_{N}} 
\newcommand{\bfP}{\mathbf{P}}

\newcommand{\bbE}{\mathbb{E}}
\newcommand{\bbZ}{\mathbb{Z}}

\title[Disordered long-range Discrete Gaussian chain]{Delocalisation and scaling limit for the disordered long-range Discrete Gaussian Chain}

\author[C. Chalhoub]{Christopher Chalhoub}
\address[C. Chalhoub]{Department of Mathematics, Imperial College London, London, United Kingdom}
\email{christopher.chalhoub21@imperial.ac.uk}

\author[P. Dario]{Paul Dario}
\address[P. Dario]{Laboratoire Analyse Géométrie Modélisation (AGM), CY Cergy Paris Universit\'e, Cergy-Pontoise, France}
\email{paul.dario@cyu.fr}

\author[C. Faipeur]{Corentin Faipeur}
\address[C. Faipeur]{Laboratoire Analyse Géométrie Modélisation (AGM), CY Cergy Paris Universit\'e, Cergy-Pontoise, France}
\email{corentin.faipeur@cyu.fr}

\author[A. Le Ny]{Arnaud Le Ny}
\address[A. Le Ny]{Laboratoire d'Analyse et de Mathématiques Appliquées, LAMA UMR CNRS 8050, UPEC, 61 avenue du Général de Gaulle, 94010 Créteil Cedex, France}
\email{arnaud.le-ny@u-pec.fr}

\begin{document}

\begin{abstract}
In this paper, we study the effect of an external random field on the localisation / delocalisation of the long-range Discrete Gaussian Chain (DGC), recently studied by Garban~\cite{garban2023invisibility} and Coquille \emph{et al.}~\cite{coquille2024quantitative}. We prove that the model is delocalised at every inverse temperature for any polynomial decay $\alpha > 3/2$. This behaviour differs qualitatively from the one of the disorder-free long-range DGC, for which delocalisation at every temperature occurs only for $\alpha > 2$, similarly to what happens for long-range Ising models as already proved by Aizenman and Wehr \cite{AW1989}. Additionally, we provide a more detailed description of the delocalised phase by identifying the growth exponent and scaling limit of the height at the origin for any decay exponent $\alpha > 3/2$.

\end{abstract}
\maketitle
\section{Introduction}

\subsection{Models and motivations}
\noindent

\medskip

\emph{Discrete Gaussian models}, also known as integer-valued Gaussian Free Fields, are models of random interfaces introduced in order to simplify the study interfaces arising in phase separation; see, for example, Minlos and Sinai \cite{Minlos-Sinai1967a}, Gallavotti \cite{Gallavotti1972}, Abraham and Reed \cite{Abraham1976}, or Velenik \cite{Ve06} for the genesis of this question for nearest neighbours (n.n.) Ising models. Obtaining a rigorous mathematical description of this phase separation
and understanding the properties of the corresponding ``interfaces'' between phases turned out to be challenging problems, which motivated many developments in equilibrium statistical mechanics, and eventually led to the introduction of new ``effective'' interface models, such as the Solid-on-Solid (SoS) or the Discrete Gaussian models, see e.g. \cite{Chui-Weeks, KH82, BLP82}.\\

A rigorous probabilistic framework for studying these systems was established through the DLR formalism of Dobrushin \cite{Dobrushin} and Lanford and Ruelle \cite{LanfordRuelle}. In this approach, microscopic states are described by spin configurations $\sigma=(\sigma_i)_{i\in\mathbb Z^d}$, and macroscopic equilibrium states are represented by probability measures specified by a family of conditional distributions with prescribed boundary conditions. In finite volume, these measures are given by the Boltzmann--Gibbs (BG) distribution
\begin{equation*}
\mu(\sigma)\propto e^{-\beta H(\sigma)},
\end{equation*}
where $\beta$ denotes the inverse temperature and $H$ is the Hamiltonian of the system. Infinite-volume Gibbs measures then arise as thermodynamic limits of these finite-volume distributions or linear combinations of them.
In contrast to the classical Kolmogorov framework for probability measures on infinite volume product (Polish) spaces, DLR formalism allows the coexistence of several Gibbs measures in infinite volume, and thus provides a rigorous mathematical description of phase transitions.\\

A central goal in mathematical statistical mechanics is then to understand the set of Gibbs measures. In this direction, a fundamental result was the discovery of Dobrushin of the existence of (infinitely many) non-translation-invariant extremal Gibbs measures (called {\em Dobrushin states}) for the low temperature Ising model in dimensions $d\ge3$ \cite{Dobrushin_BC}. These measures arise for instance as thermodynamic limits of finite-volume Ising models with so called  $\pm$ Dobrushin boundary conditions and are associated with \emph{rigid} (or localised) interfaces separating the two pure phases $\mu^+$ and $\mu^-$ (N.B. These states are obtained as the thermodynamic limit of finite-volume Ising models with plus and minus boundary conditions). In contrast, the two-dimensional Ising model exhibits \emph{rough} (or delocalised) interfaces at low temperature. Aizenman~\cite{aizenman1980translation} and Higuchi~\cite{higuchi1979absence} obtained a complete description of its Gibbs measures by showing that every such measure is a convex combination of the two pure phases $\mu^+$ and $\mu^-$. In particular, this prevents the existence of Dobrushin states for $d=2$, and the interface between the plus and minus phases fluctuates under $\pm$ Dobrushin boundary condition. Quantifying this fluctuation has motivated the introduction of effective interface models of height functions in dimension $d-1$, such as the discrete Gaussian model.

\medskip
In this article, we study a one-dimensional long-range discrete Gaussian model of height functions $\phi \in \bbZ^\bbZ$ with a random external field, whose Hamiltonian is formally given by
\begin{equation*}
H(\phi)=\sum_{x\neq y\in\mathbb Z}\dfrac{\bigl(\phi(x)-\phi(y)\bigr)^2}{|x-y|^{\alpha}} + \sum_{x \in \mathbb{Z}} \eta_x \phi(x),
\end{equation*}
where $(\eta_x)_{x \in \mathbb{Z}}$ are i.i.d. centered random variables and the parameter $\alpha$ controls the decay of the interaction. Varying $\alpha$ changes the large-scale behaviour of the model and provides a way to interpolate between different effective dimensions. The Imry--Ma phenomenon~\cite{IM75} (rigorously established by Aizenman and Wehr~\cite{AW1990}) states that the incorporation of a quenched disorder in the form of a random external field can substantially alter the properties of spin systems in low dimensions. Our goal is to investigate how long-range interactions and quenched disorder jointly affect the behaviour of the model, in particular regarding the localisation/delocalisation properties and its scaling limit in the delocalised regime.\\

For general models of random interface, the notions of localisation and delocalisation are usually expressed in equivalent (or closely related) ways. On the one hand, localisation and delocalisation are expressed within the DLR framework through the existence, respectively the absence, of infinite-volume Gibbs measures.
On the other hand, they are often characterised by the statement that some $\mathbb{L}^p$-norm of the height at the origin $\phi(0)$ is bounded or unbounded.

In the presence of quenched disorder, these different perspectives become more delicate from a measure-theoretic point of view, and we refer to the works of Bovier and Külske~\cite{bovier1992stability, bovier1993hierarchical, BK94, BK96} in the 1990s and by Cotar and Külske~\cite{CK12, CK15} for a detailed investigation of the existence, uniqueness, and structure of random (gradient) Gibbs measures in this setting.
We also refer to the monographs of Sheffield \cite{Sheffield} and Velenik \cite{Ve06} for additional information on random interfaces (without disorder).\\

In the present work, we focus on two complementary notions of delocalisation of the interface for decay exponents $\alpha>3/2$. The first result, called {\em qualitative delocalisation}, establishes the absence of \emph{shift-covariant} random Gibbs measures (see Theorem~\ref{thm:qualitative-delocalisation}), while the second one, called {\em quantitative delocalisation}, provides the asymptotic growth of the variance of the height at a the site $0 \in \mathbb{Z}$ (in finite-volume systems with Dirichlet boundary condition, see Theorem~\ref{thm:deloc}).
For smaller values of $\alpha$, one expects localisation to hold (by comparison with the real-valued disordered Gaussian chain, see Table~\ref{table1.1}). Establishing this result is beyond the scope of this article and would likely require a more precise analysis of the infinite-volume Gibbs measures within the DLR framework; see Cotar and Külske~\cite{CK12, CK15} for related results in disordered interface models.

\subsection{Framework and results}
Our framework is thus the one of unbounded one-dimensional spin systems with polynomially decaying interaction.

\subsubsection{Height functions} We consider random (discrete) height functions $\phi=(\phi(x))_{x \in \bbZ}$ on the one-dimensional lattice $\mathbb Z$. For $N\in \mathbb N$, we denote by
\[
    \Lambda_N := \{-N,\ldots,N\} \subseteq \mathbb Z,
\]
and consider the height functions $\phi$ in finite volume $\Lambda_N$ with Dirichlet boundary condition to be the  
functions $\phi:\mathbb Z\to \mathbb Z$ such that $\phi(x)=0$ for all
$x\notin \Lambda_N$. We introduce the corresponding finite-volume configuration space 
\[
    \Omega_N := \{\phi:\mathbb Z\to\mathbb Z \mid \phi(x)=0 \text{ for all }
    x\notin \Lambda_N\}
    \simeq \mathbb Z^{\Lambda_N}.
\]
\subsubsection{Long-range Hamiltonian and random field} Let $\alpha>1$ be a decay exponent. For $\phi\in\Omega_N$, define the zero-boundary long-range quadratic Hamiltonian by
\begin{equation*}
    \mathcal H_{N,\alpha}^0(\phi):=
    \frac12 \sum_{\substack{\{x,y\}\subset\mathbb Z,\\ x\ne y}} \dfrac{(\phi(x)-\phi(y))^2}{|x-y|^{\alpha}},
\end{equation*}
Since $\phi$ vanishes outside $\Lambda_N$, the sum is finite. We similarly define the zero-boundary nearest-neighbour quadratic Hamiltonian by
\begin{equation*}
    \mathcal H_{N,\infty}^0(\phi):=
    \frac12 \sum_{\substack{\{x,y\}\subset\mathbb Z, \\ |x- y| = 1}} (\phi(x)-\phi(y))^2.
\end{equation*}
To introduce disorder, let $(\Xi, \Sigma, \mathbb P)$ be a probability space supporting an i.i.d. collection of random variables $\eta=(\eta_x)_{x\in\mathbb Z}$ with finite second moment such that
\begin{equation*}
    \bbE \left[ \eta_x \right] = 0 ~~\mbox{and}~~ \bbE \left[ \eta_x^2 \right] = 1,
\end{equation*}
where $\mathbb E$ denotes the expectation with respect to the
disorder law $\mathbb P$.

\subsubsection{Disordered Gibbs measures}
 We then let $\lambda > 0$ be a disorder strength, $\beta>0$ be an inverse temperature, and, for any fixed realisation of the disorder $\eta$ and any decay exponent $\alpha \in (1 , \infty]$, we define the finite-volume disordered Hamiltonian by
\begin{equation} \label{eq:defdisorderedHamiltonian}
   \forall \phi\in\Omega_N, \hspace{3mm}  \mathcal H_{N, \alpha,\lambda}^\eta(\phi):=
    \mathcal H_{N,\alpha}^0(\phi)-
    \lambda \sum_{x\in\Lambda_N} \eta_x \phi(x).
\end{equation}
Equipped with these definitions, we now introduce the central object of study of this article: the disordered Gibbs measure.

\begin{definition}[Disordered Gibbs measure]
For any decay exponent $\alpha \in (1, \infty]$, any disorder strength $\lambda > 0$, any inverse temperature $\beta>0$ and any realisation of the disorder $\eta$, we define the disordered Gibbs measure by the identity
\begin{equation} \label{eq:defGibbsmeasure}
    \forall \phi\in\Omega_N, \hspace{3mm} \Gibbs(\phi):=
    \frac{1}{Z_{N, \alpha, \lambda, \beta}^{\eta}}
    \exp\bigl(-\beta \mathcal H_{N, \alpha, \lambda}^\eta(\phi)\bigr),
\end{equation}
where the partition function is
\[
    Z_{N, \alpha, \lambda, \beta}^{\eta}:=\sum_{\phi\in\Omega_N}
    \exp\bigl(-\beta \mathcal H_{N, \alpha, \lambda}^\eta(\phi)\bigr).
\]
We denote by $\EE$ the expectation with respect to $\Gibbs.$
\end{definition}

\begin{remark}
 For a realisation of the disorder $\eta$, we call any configuration that minimises the Hamiltonian~\eqref{eq:defdisorderedHamiltonian} a \textit{ground state for the disorder $\eta$}. Due to the integer-valued constraint, the ground state is not necessarily unique for a fixed realisation of the disorder. However, the set of ground states is always finite and uniqueness holds for Lebesgue-almost every disorder realisation $\eta \in \mathbb{R}^{\Lambda_N}$. The ground states are related to the Gibbs measure~\eqref{eq:defGibbsmeasure} as follows: for any realisation of the disorder $\eta$, the measure $\Gibbs$ converges as $\beta \to \infty$ to the uniform measure on the ground states for $\eta$. We denote this measure by $\mu^{\eta}_{N,\alpha, \lambda, \infty}$. 
\end{remark}

\subsubsection{Results}
We first give a qualitative form of delocalisation, formulated as the non-existence of shift-covariant infinite-volume random Gibbs measures with finite annealed first moment when $\alpha>3/2$ (see Definition~\ref{def:def3.1}). We will make use of the notation
\begin{equation} \label{def:defOmega}
    \Omega_\alpha := \left\{ \phi : \bbZ \to \bbZ \mid \sum_{y \in \bbZ \setminus \{0 \}} |y|^{-\alpha} \left| \phi(y) \right| < \infty \right\}.
\end{equation}

\begin{theorem}[Qualitative delocalisation]
\label{thm:qualitative-delocalisation}
Let \(\alpha \in (3/2,\infty]\), \(\lambda>0\), and \(\beta>0\). There does not exist a measurable map
\[
    \eta \longmapsto \mu^\eta \in \mathcal P(\Omega_\alpha)
\]
such that the following two properties hold:
\begin{enumerate}
    \item $\eta \mapsto \mu^\eta$ is a shift-covariant infinite-volume Gibbs measure for the disordered long-range Gaussian chain with parameters \((\alpha,\lambda,\beta)\) (see Definition~\ref{def:def3.1});
    \item it has finite annealed first moment:
    \[
        \mathbb{E}\bigl[\mu^\eta[|\phi(0)|]\bigr]<\infty .
    \]
\end{enumerate}
\end{theorem}
\begin{remark}
    The assumption that $\mu^\eta$ is a probability measure on the set $\Omega_\alpha$, rather than on the space of functions from $\mathbb{Z}$ to $\mathbb{Z}$, is a technical requirement that is important for the definition of Gibbs measures (see Definition~\ref{def:3.1} and Remark~\ref{remark3.2} below).
\end{remark}

We then provide a quantitative version of the delocalisation for the disordered height function in the same regime $\alpha > 3/2$.
We identify the asymptotic behaviour of the $\mathbb{L}^2$-norm of the random variable $\phi(0)$ under the annealed distribution $$\mathbb P_{N,\alpha,\lambda, \beta}^{\mathrm{ann}}(d\eta,d\phi):=\mathbb P(d\eta)\,\Gibbs(d\phi)$$ as $N \to \infty$. As a by-product of the proof, we also obtain an upper bound in the case $\alpha \leq 3/2$ (which is not expected to be sharp).
In order to state the result, we need to introduce the Green's function associated with the fractional Laplacian in the interval $[-1 , 1]$ with Dirichlet boundary condition.

\begin{definition}[Green's function for the Laplacian and fractional Laplacian]
We introduce the following functions:
\begin{itemize}
    \item \underline{Green's function:} We let $G: [-1,1] \to \R$ be the solution of the equation
\begin{equation*}
    \left\{ \begin{aligned}
        -\Delta G & = \delta_0 &~\mbox{in}~& [-1,1], \\
        G & = 0 &~\mbox{on}~& \partial [-1,1].
    \end{aligned} \right.
\end{equation*}
    Here, $\Delta$ denotes the classical Laplacian (i.e., the second derivative in dimension $1$).
    \item \underline{Fractional Green's function:} For a decay exponent $\alpha \in (3/2 , 3)$, we let $G_\alpha:\R \to \R$ be the solution of the equation
\begin{equation*}
    \left\{ \begin{aligned}
        (-\Delta)^{\frac{\alpha - 1}{2}} G_{\alpha} & = \delta_0 &~\mbox{in}~& [-1,1] \\
        G_{\alpha} & = 0 &~\mbox{on}~& \R \setminus [-1,1].
    \end{aligned} \right.
\end{equation*}
    Here, for all $s\in(0,1)$, $(-\Delta)^s$ is the fractional Laplacian (see e.g. \cite{LSSW17}).
\end{itemize}
\end{definition}

\begin{remark} \label{rem:remark1.3}
    For the one-dimensional Green's function with Dirichlet boundary condition, we have the explicit formula
    \begin{equation*}
        \forall x \in [-1,1], \hspace{3mm} G(x) = \frac{(1 - |x|)}{2}.
    \end{equation*}
\end{remark}

\begin{theorem}[Delocalisation and asymptotic typical height at the origin]\label{thm:deloc}
 For any decay exponent $\alpha \in (1 , \infty]$, any inverse temperature $\beta > 0$ and any disorder strength $\lambda > 0$, one has the following convergences:
 \begin{itemize}
     \item For $\alpha \in (3/2, 3)$,
     \begin{equation*}
        \frac{1}{N^{2\alpha - 3}} \mathbb E\bigl[\EE[ \phi(0)^2 ]\bigr] \underset{N \to \infty}{\longrightarrow} \lambda^2 \int_{-1}^1 G_{\alpha}(x)^2 \, dx.
     \end{equation*}
     \item For $\alpha = 3$, 
      \begin{equation*}
        \frac{(\ln N)^2}{N^{3}} \mathbb E\bigl[\EE[ \phi(0)^2 ]\bigr] \underset{N \to \infty}{\longrightarrow} \lambda^2 \int_{-1}^1 G(x)^2 \, dx.
     \end{equation*}
     \item For $\alpha \in (3, \infty]$,
      \begin{equation*}
        \frac{1}{N^{3}} \mathbb E\bigl[\EE[ \phi(0)^2 ]\bigr] \underset{N \to \infty}{\longrightarrow} \frac{\lambda^2}{\sigma_{\alpha}^2} \int_{-1}^{1} G(x)^2 \, dx ~~\mbox{with} ~~\sigma_\alpha :=  \sum_{k =1}^\infty \frac{1}{k^{\alpha - 2}}.
     \end{equation*}
 \end{itemize}
 Moreover, for $\alpha\in(1,3/2]$, there exists a constant $C:=C(\alpha,\lambda, \beta) < \infty$ such that
 \begin{equation*}
     \forall N \geq 2, \qquad \bbE\bigl[\EE[ \phi(0)^2 ]\bigr] \leq C (\ln N)^2.
 \end{equation*}
\end{theorem}

\begin{remark}
Let us make a few remarks about this result:
\begin{itemize}
\item The limiting integral can be explicitly computed in the case $\alpha \in [3 , \infty]$ using Remark~\ref{rem:remark1.3} and is equal to $1/6$.
\item These results hold for the measure $\mu^{\eta}_{N,\alpha, \lambda, \infty}$ (N.B. The proof is in fact simpler in this case and we refer to Section~\ref{sec:sketchofproof} for a sketch of the argument).
\item This theorem only requires the disorder to consist of independent random variables with mean $0$ and variance $1$ (i.e., we do not need the random variables $(\eta_x)_{x \in \bbZ}$ to be identically distributed).
\item In the case of the nearest-neighbour model (i.e., $\alpha = \infty$), we have $\sigma_\infty = 1$.
\item The same results hold if we replace the squared $\mathbb{L}^2$-norm by the annealed variance , since $\bbE\bigl[\EE[ \phi(0) ]\bigr]^2$ is asymptotically negligible compared to $\bbE\bigl[\EE[ \phi(0)^2 ]\bigr]$, by \eqref{eq:sumandvarianceiid} and
Proposition~\ref{prop:distanceheighttogroundstate}.
\end{itemize}
\end{remark}

The last main result of this article refines the previous one by showing a central limit theorem for the random variable $\phi(0)$ when $\alpha > 3/2$. Combined with Theorem~\ref{thm:deloc}, it identifies the scaling limit of the height $\phi(0)$.

\begin{theorem}[Central limit theorem] \label{thm:thm1.4TCL}
For any decay exponent $\alpha \in (3/2, \infty],$ any inverse temperature $\beta > 0$, and any disorder strength $\lambda > 0$,  if we let $\phi_N$ be distributed according to $\mathbb P_{N,\alpha,\lambda, \beta}^{\mathrm{ann}}$, then
\begin{equation*}
    \frac{\phi_N(0)}{\sqrt{\mathbb E\bigl[\EE[ \phi(0)^2 ]\bigr]}} \overset{\mathcal{L}}{\underset{N \to \infty}{\longrightarrow}} \mathcal N(0,1).
\end{equation*}
\end{theorem}
\begin{remark} Let us make two remarks about this result: \noindent
    \begin{itemize}
        \item As for the previous theorem, this result stays true if one replaces the $\mathbb L^2$-norm by the square root of the annealed variance in the denominator.
        \item Unlike Theorem~\ref{thm:deloc}, this result requires the random variables $(\eta_x)_{x \in \mathbb{Z}}$ to be i.i.d.
    \end{itemize}
\end{remark}
Let us also define the real-valued long-range Gaussian chain with disorder $\eta$ to be the probability distribution on $\Omega_N^{\mathbb R} := \{u:\mathbb Z\to\mathbb R \mid \ u(x)=0 \text{ for all }x\notin\Lambda_N\}$ given by
\begin{equation} \label{eq:defreal-valuedmodel}
    \mu_{N, \alpha, \lambda, \beta}^{\eta, \R}(du)
    :=
    \frac{1}{Z_{N, \alpha, \lambda, \beta}^{\eta, \R}}
    \exp\left(
            -\beta \mathcal H_{N, \alpha, \lambda}^\eta(u)
    \right)
    \prod_{x\in\Lambda_N}du(x),
\end{equation}
where $Z^{\eta, \R}_{N, \alpha, \lambda, \beta}$ is the normalising constant. The model~\eqref{eq:defreal-valuedmodel} is a multivariate normal distribution which can be studied directly by identifying its mean and covariance matrix (see Section~\ref{sec:sectionrealvaluedgaussian} below). In particular, it is not difficult to show that Theorem~\ref{thm:deloc} and Theorem~\ref{thm:thm1.4TCL} also hold for the real-valued model~\eqref{eq:defreal-valuedmodel}. The integer-valued and real-valued models exhibit the same large-scale behaviour; this is an example of \emph{the invisibility of the integers} (following the terminology of~\cite{garban2023invisibility}). The same phenomenon was established by Garban~\cite{garban2023invisibility} for the disorder-free Gaussian chain in the high-temperature regime; see Section~\ref{sec:sectionrelatedwork}.\\

Our  Localisation/Delocalisation results for the disordered long-range Gaussian chain are summarised in the following table.\footnote{Note that Theorem~\ref{thm:deloc} contains strictly more information than Table 1 as  for $\alpha > 3/2$ the asymptotic behaviour of the $\mathbb L^2$-norm of the height is also identified.}

\medskip

\renewcommand{\arraystretch}{1.5}
\begin{center}
\begin{tabular}{llllll}
	\hline
	\multicolumn{1}{|l|}{$\bbE\bigl[\EE[ \phi(0)^2 ]$} & 	\multicolumn{1}{|l|}{$\mathbb{Z}$-valued}& \multicolumn{1}{l|}{$\R$-valued}  \\ \hline   
    \multicolumn{1}{|l|}{$\alpha\in (1 , 3/2)$} & \multicolumn{1}{|l|}{$\leq (\ln N)^2$} & \multicolumn{1}{l|}{$1$} \\ \hline
    \multicolumn{1}{|l|}{$\alpha = 3/2$} & \multicolumn{1}{|l|}{$\leq (\ln N)^2$} & \multicolumn{1}{l|}{$\ln N$} \\ \hline
	\multicolumn{1}{|l|}{$\alpha \in (3/2 , 3)$} & \multicolumn{1}{|l|}{$N^{2\alpha - 3}$ } & \multicolumn{1}{l|}{$N^{2\alpha - 3}$} \\ \hline
	\multicolumn{1}{|l|}{$\alpha=3$} &\multicolumn{1}{|l|}{$N^3/(\ln N)^2$} & \multicolumn{1}{l|}{$N^3/(\ln N)^2$} \\ \hline        
	\multicolumn{1}{|l|}{$\alpha>3$} &\multicolumn{1}{|l|}{$N^3$} & \multicolumn{1}{l|}{ $N^3$} \\ \hline                 
\end{tabular}
\captionof{table}{Localisation/Delocalisation for the disordered long-range Gaussian chain (N.B. the results hold for any $\beta > 0$ and any $\lambda >0$). \label{table1.1}}
\end{center}

\begin{figure}
    \centering
    \includegraphics[width=1\linewidth]{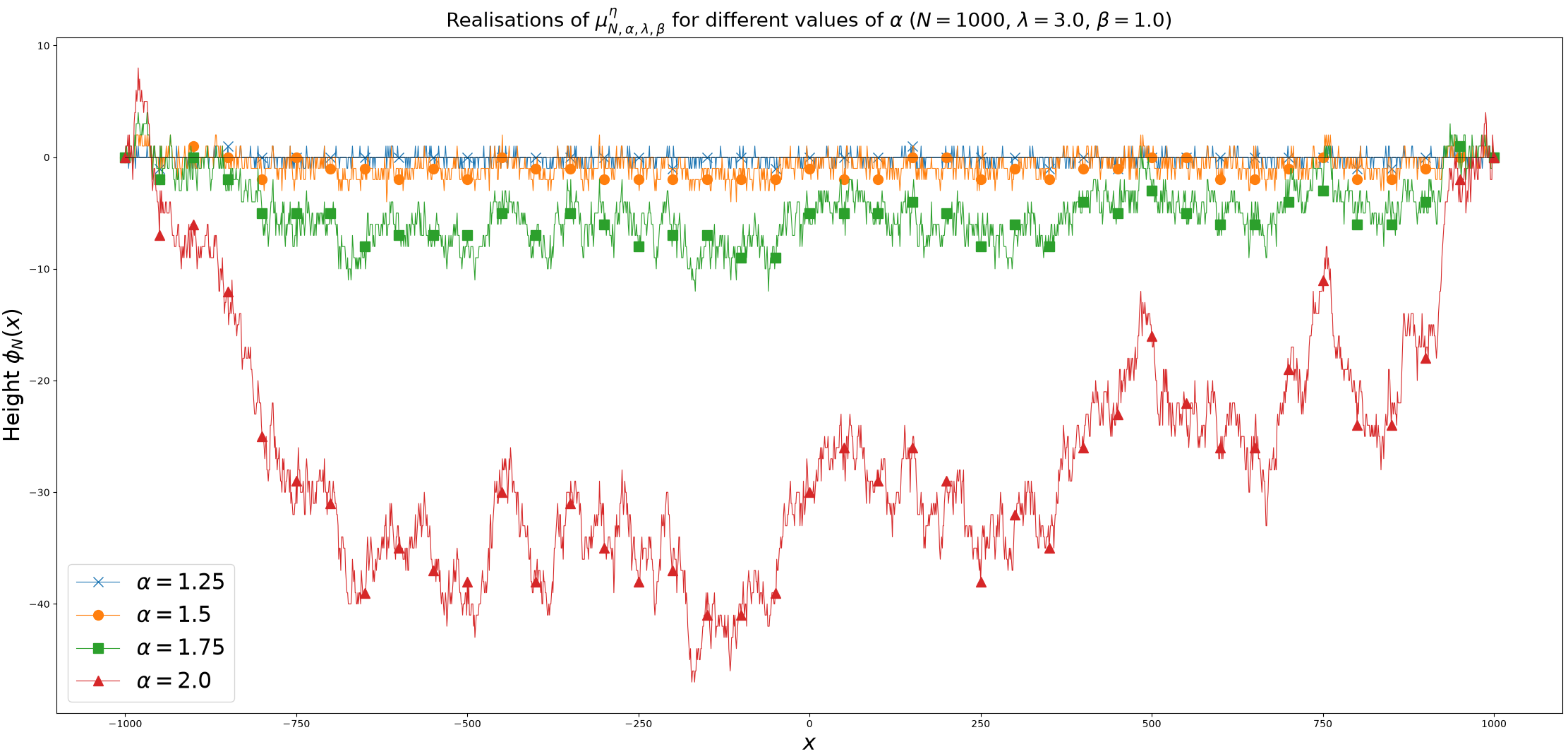}
    \caption{Four simulations with the Metropolis--Hastings algorithm of the disordered long-range DGC with the same realisation of the disorder (here i.i.d. standard Gaussians) and different values of the decay exponent $\alpha$.}
    \label{fig:several_alpha}
\end{figure}

\subsection{Related works} \label{sec:sectionrelatedwork}

This section presents an overview of related results.

\medskip

\textit{The long-range Gaussian chain.} 

\medskip

The long-range Gaussian chain without disorder (i.e., the model~\eqref{eq:defGibbsmeasure} with $\lambda = 0$) exhibits a different behaviour than its disordered counterpart and has been studied on its own. For the non-disordered model, the threshold for the localisation/delocalisation is the decay exponent $\alpha = 2$, where the model is known to exhibit a \emph{roughening} phase transition from localisation to delocalisation when the temperature increases. More precisely, the non-disordered counterpart of Table~\ref{table1.1} is given by the following table (N.B. In all the cases, upper and lower bounds on the variance of the height $\phi(0)$ are known, matching up to a multiplicative constant).

\bigskip

\begin{center}
\begin{tabular}{llllll}
	\hline
	\multicolumn{1}{|l|}{$E_{N , \beta , \alpha}[ \phi(0)^2 ]$} & 	\multicolumn{1}{|l|}{$\mathbb{Z}$-valued, $\beta \gg 1$}& \multicolumn{1}{l|}{$\mathbb{Z}$-valued, $\beta \ll 1$} &  \multicolumn{1}{l|}{$\mathbb{R}$-valued} \\ \hline   
    \multicolumn{1}{|l|}{$\alpha\in (1 , 2)$} & \multicolumn{1}{|l|}{$1$ } & \multicolumn{1}{l|}{$1$} & \multicolumn{1}{l|}{$1$} \\ \hline
	\multicolumn{1}{|l|}{$\alpha=2$} & \multicolumn{1}{|l|}{$1$ } & \multicolumn{1}{l|}{$\ln N$} & \multicolumn{1}{l|}{$\ln N$} \\ \hline
	\multicolumn{1}{|l|}{$\alpha\in(2,3)$} & \multicolumn{1}{|l|}{$N^{\alpha-2}$} &\multicolumn{1}{l|}{$N^{\alpha-2}$} &\multicolumn{1}{l|}{$N^{\alpha-2}$} \\ \hline
	\multicolumn{1}{|l|}{$\alpha=3$} &\multicolumn{1}{|l|}{$N/\ln N$} & \multicolumn{1}{l|}{$N/\ln N$}  & \multicolumn{1}{l|}{$N/\ln N$}  \\ \hline        
	\multicolumn{1}{|l|}{$\alpha>3$} &\multicolumn{1}{|l|}{$N$} & \multicolumn{1}{l|}{ $N$} & \multicolumn{1}{l|}{ $N$} \\ \hline                 
\end{tabular}
\captionof{table}{Localisation/Delocalisation for the long-range Gaussian chain without disorder (in the third column, the results hold for any $\beta >0$).\label{table1}}
\end{center}
\medskip

\begin{figure}
    \centering
    \begin{subfigure}[b]{\linewidth}
        \includegraphics[width=\linewidth]{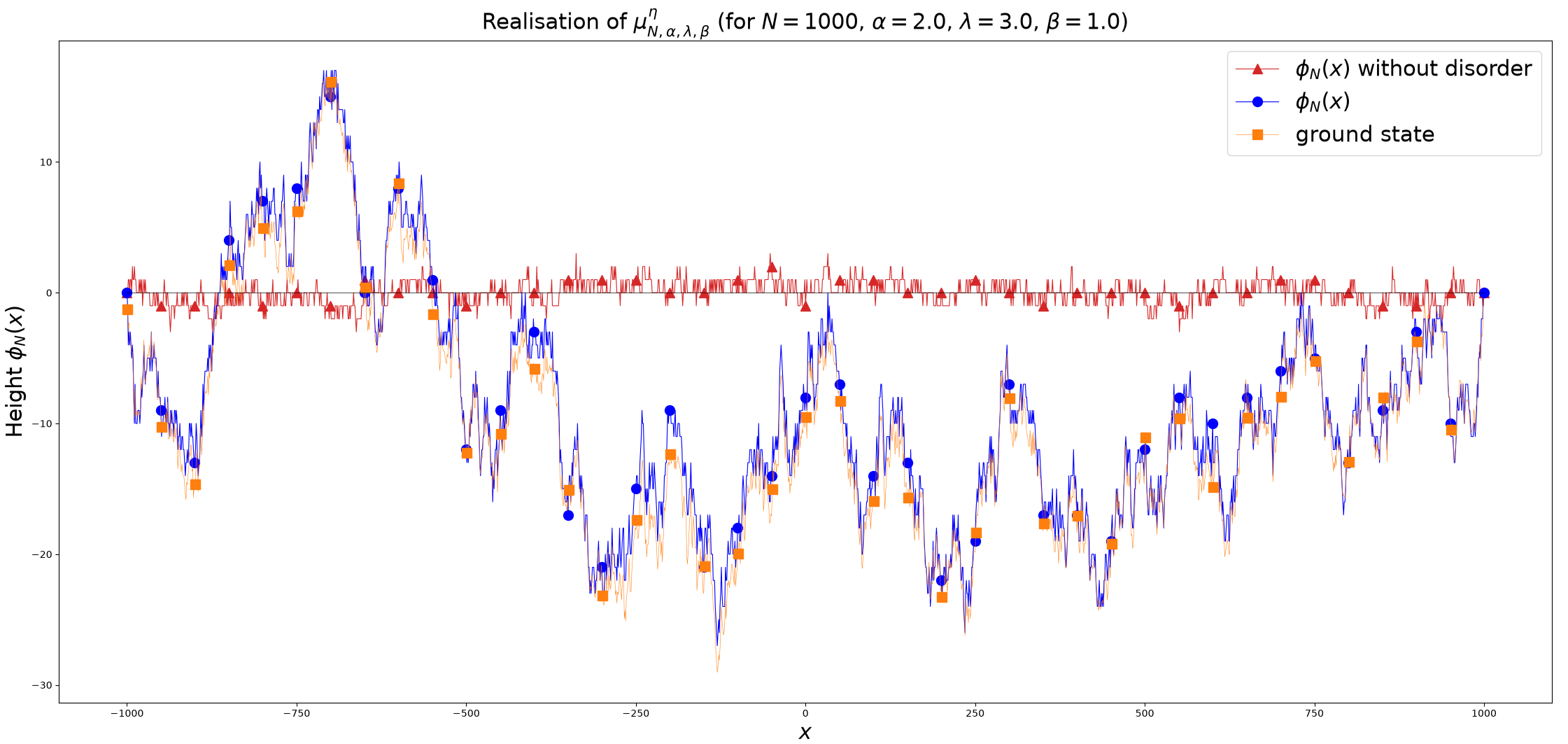}
        \caption{At $\alpha=2$, the typical height of the DGC at the origin is at most $\sqrt{\ln N}$, whereas for the disordered version it is of order $\sqrt N$; in accordance with Proposition \ref{prop:distanceheighttogroundstate}, the distance between the disordered DGC and the real-valued ground state is logarithmically small.}
    \end{subfigure}

    \vspace{5mm}
    \begin{subfigure}[b]{\linewidth}
        \includegraphics[width=\linewidth]{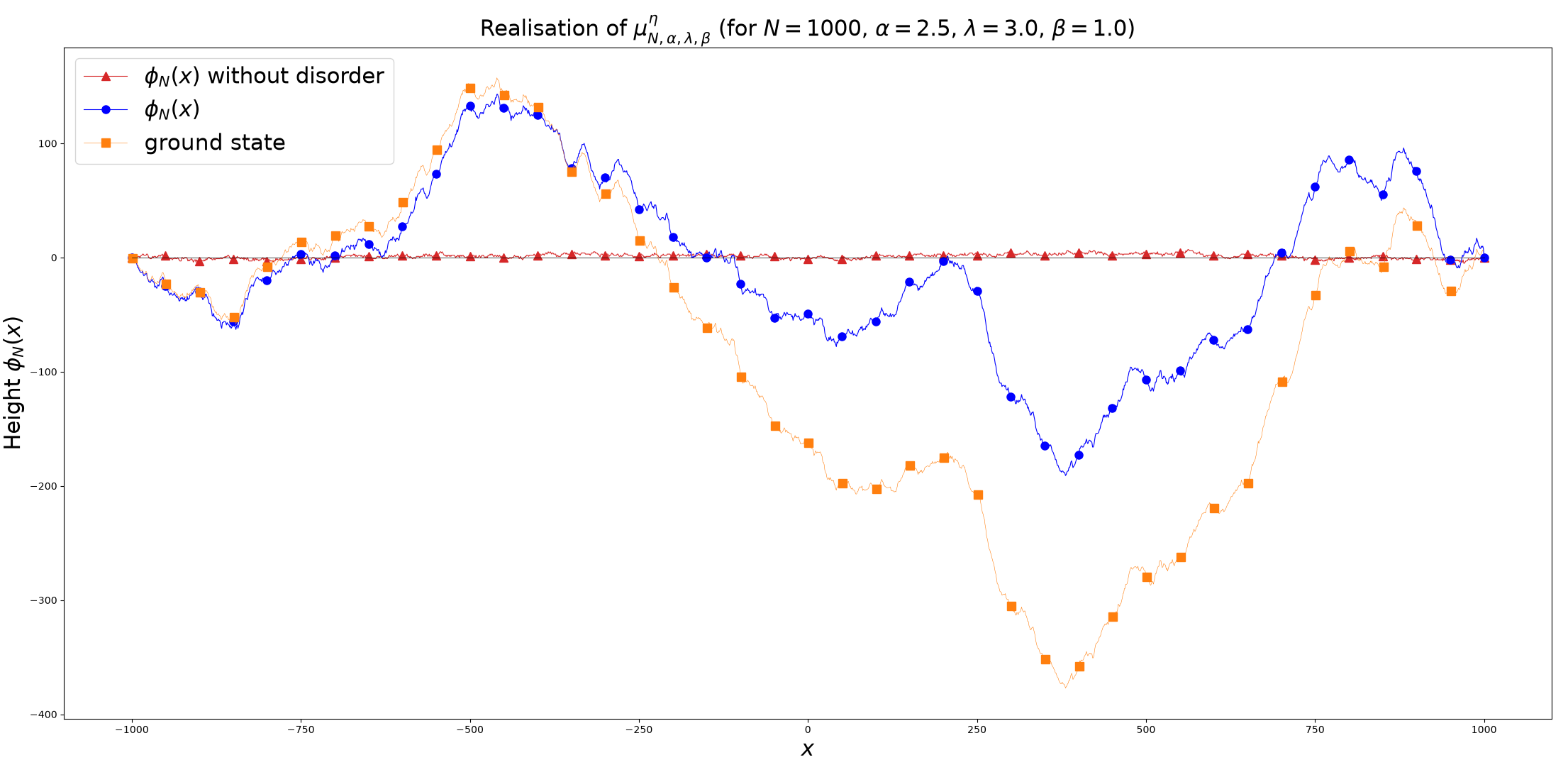}
        \caption{At $\alpha=2.5$, both the DGC and the disordered DGC delocalise, but on different scales (the typical height is of order $N^{1/4}$ for the former and of order $N$ for the latter); in accordance with Proposition \ref{prop:distanceheighttogroundstate}, the disordered DGC can deviate from the real-valued ground state by approximately $\sqrt N$.}
    \end{subfigure}
    \caption{Simulations with the Metropolis--Hastings algorithm of the long-range DGC (in red, triangles) and disordered DGC (in blue, circles), with decay exponent $\alpha=2$ (top) and $\alpha=2.5$ (bottom).
    We have also represented (in orange, squares) the real-valued ground state of the model (see Section \ref{sec:sketchofproof} or \ref{sec:sectionrealvaluedgaussian}). }
    \label{fig:with_or_without_you_eta}
\end{figure}

These results were proved in a series of works. The case of the critical decay exponent $\alpha = 2$ was first studied by Kjaer and Hilhorst~\cite{KH82} who established the logarithmic delocalisation for the chain with periodic boundary conditions for $\beta = 1$ when the long-range interaction is $\frac{1}{|i - j|^2 - 1/4}$. Still for the decay exponent $\alpha = 2$, Fr\"{o}hlich and Zegarlinski~\cite{frohlich1991phase} showed the localisation of the chain at sufficiently low temperature by adapting the Peierls argument for the long-range Ising model of~\cite{FrohlichSpencer}. For the decay exponents $\alpha < 2$, the localisation at every inverse temperature was established by Garban~\cite{garban2023invisibility} (using the correlation inequality of~\cite{aizenman2021depinning, regev2017inequality}). In the regime $\alpha \geq 2$, Garban showed the delocalisation of the chain for sufficiently small $\beta$, and even identified the scaling limit of the chain, showing that the model exhibits invisibility of the integers (N.B. the same invisibility phenomenon is observed on the disordered Gaussian chain but the incorporation of the disorder has the effect of changing the universality class of the model). In the low temperature regime, the delocalisation was proved qualitatively in~\cite{coquille2024absence} and quantitatively in~\cite{coquille2024quantitative}.

\medskip

\textit{Disordered spin system and the Imry--Ma phenomenon}

\medskip

One of the motivations to study disordered spin system is the Imry--Ma phenomenon. This phenomenon was predicted by Imry and Ma~\cite{IM75}, rigorously proved by Aizenman and Wehr~\cite{AW1989, AW1990}, and asserts that the incorporation of a disorder in the form of a random external field can significantly alter the qualitative properties of spin systems by causing a disappearance of the first-order phase transitions in low dimensions. The same phenomenon appears on the Gaussian chain where the addition of a quenched disorder changes the localisation/delocalisation and scaling limit of the model (see Tables~\ref{table1.1} and~\ref{table1}).

The results of Aizenman and Wehr are qualitative, and the question of the quantification of the Imry--Ma phenomenon was first addressed for the two-dimensional random-field Ising model. In this direction, Chatterjee~\cite{C18} obtained an upper bound at the rate $(\ln \ln L)^{-1/2}$ on the effect of boundary conditions on the magnetisation at the center of a box of side length~$L$. A polynomial rate was then obtained by Aizenman and Peled~\cite{AP19} and exponential decay was obtained first by Ding and Xia~\cite{ding2019exponential} at zero temperature and then by Ding and Xia~\cite{ding2021exponential} and by Aizenman, Harel and Peled~\cite{AHP20} at every temperature (we also refer to ~\cite{FI84, B85, DKP95, CJN18} for related results in the regime of high temperature or strong disorder).

In the one dimensional case, a detailed study of the random field Ising chain (e.g. domain wall structure and behaviour of the free energy in the large interaction limit) was undertaken by Collin, Giacomin and Hu~\cite{CGH2, AHLCGH}, Collin~\cite{CollinALEA, collin2025further} and Collin, Giacomin, Greenblatt and Hu~\cite{collin2026lyapunov}.

In dimension $d \geq 3$, the behaviour is of a different nature and the model exhibits a phase transition in the temperature when the strength of the disorder is sufficiently small. In this line, Imbrie~\cite{I85} proved that the ground state of the random field Ising model exhibits long-range order when the strength of the field is sufficiently small, and this result was extended to the low-temperature regime by Bricmont and Kupiainen~\cite{BK87L, BK88} by using a sophisticated renormalisation group argument. Recently, Ding and Zhuang~\cite{DingZhuang} obtained a new and simple proof of the result of~\cite{BK87L, BK88} and extended it to the $q$-states Potts model, and Ding, Liu and Xia~\cite{ding2024long} proved the existence of long-range order for any inverse temperature $\beta > \beta_c$ (for sufficiently small disorder strength).

Additional background on the random-field Ising model can be found in~\cite[Chapter 7]{Bo06} and~\cite{N98T}.

\medskip

\textit{Long-range spin systems with and without disorder.}

\medskip

For the long-range Ising model in one dimension, \cite{dobrushin_uniqueness} and Ruelle~\cite{ruelle1968statistical} proved that there is no phase transition when the decay exponent satisfies $\alpha > 2$, and Dyson~\cite{Dyson69} showed the existence of such a phase transition for $\alpha \in (1 , 2)$. In the case of the critical decay exponent $\alpha = 2$, the existence of a phase transition was established by Fr\"{o}hlich and Spencer~\cite{FrohlichSpencer} (N.B. The phenomenology shares some similarities with the one of the long-range Gaussian chain for which the critical decay exponent for the localisation/delocalisation is also $\alpha = 2$).

For the long-range random field Ising model, the results of Aizenman and Wehr~\cite{AW1990} show that, for any decay exponent $\alpha \geq 3/2$, the model does not exhibit a phase transition in the temperature for any disorder strength. For the decay exponents $\alpha \in ( 3 - \ln 3/ \ln 2  , 3/2)$, Cassandro, Orlandi and Picco~\cite{cassandro2009phase} showed that the model exhibits long-range order at sufficiently low temperature and for sufficiently small disorder strength. The result was recently extended by Ding, Huang and Maia~\cite{ding2024phase} to any decay exponent $\alpha \in (1 , 3/2)$.

In higher dimension, the nearest-neighbour Ising model already presents a phase transition, so its long-range version inherits it.
The contribution~\cite{ding2024phase} shows the existence of a phase transition in dimension 2 for the random field Ising model with decay exponents $\alpha \in (2 , 3]$, while it was already ruled out by \cite{AW1990} in the case $\alpha>3$.
The work of Affonso, Bissacot and Maia~\cite{affonso2023phase} establishes the existence of long-range order in dimension $d \geq 3$ for any $\alpha \in (d , \infty)$.

\medskip
\textit{Interface models with a disorder:} 

\medskip

Our article studies the Imry--Ma phenomenon in the context of random interfaces. This line of investigation was initiated by Bovier and K\"{u}lske~\cite{BK94, BK96}, who studied a related model integer-valued random interfaces on $\mathbb{Z}^d$ with $d \geq 1$ and with a Hamiltonian taking the form
\begin{equation} \label{eq:modelBK}
    H(\phi) := \sum_{x\sim y} |\phi(x) - \phi(y)| - \lambda \sum_{k \in \mathbb{Z}} \sum_{x} \eta_x (k) \mathbf 1_{\{ \phi(x) = k \}},
\end{equation}
where $(\eta_x(k))_{x\in\mathbb Z^d, k\in\mathbb{Z}}$ are i.i.d. random variables. In~\cite{BK94} and under some suitable assumptions on the law of the disorder, the non-existence of shift-covariant Gibbs measures in dimension $d \leq 2$ was shown by adapting the argument of~\cite{AW1990}. In~\cite{BK96}, building on the renormalisation group approach of~\cite{BK88}, they proved the existence of infinite-volume Gibbs measures in dimensions $d \geq 3$ when the strength of the disorder and the temperature are sufficiently small. Recent developments on a variant of the model~\eqref{eq:modelBK} (where the interface~$\phi$ is real valued and the absolute value potential is replaced by the square potential) include the contributions of Dembin, Elboim, Hadas and Peled~\cite{dembin2024minimal, dembin2025minimal} which establish bounds on the geometric and energetic fluctuations of the interface. We additionally refer to~\cite{dembin2024minimal} for a detailed review of the literature on the topic.

In a related context, Cotar and Buchholz~\cite{buchholz2023aizenman} showed the uniqueness of shift-covariant gradient Gibbs measure for a class of random interfaces with a disorder taking the form of a random annealed Gaussian potential. We also refer to the recent contribution of Buchholz, Cotar and Schweiger~\cite{buchholz2026gradient} for additional results in this line of research.

Closer to the model~\eqref{eq:defGibbsmeasure} are the models of random interfaces with a random field taking the form 
\begin{equation} \label{mod.randomsurface}
    \mathbb{P}(d \phi) := \frac{1}{Z} \exp \left( - \sum_{x \sim y} V \left( \phi(x) - \phi(y) \right) + \lambda \sum_{x}  \eta_{x} \phi(x) \right) \prod_x d \phi(x) ,
\end{equation}
where $V : \R \mapsto \R$ is a function growing sufficiently fast at infinity. Various qualitative and quantitative aspects of these models have been studied in various contributions of Cotar, van Enter, K\"{u}lske and Orlandi~\cite{KO06, KO08, VK08, CK12, CK15}. The closest result to Theorem~\ref{thm:deloc} for the model~\eqref{mod.randomsurface} is the one of~\cite{dario2023random} which identifies the typical height of the interface in any dimension $d \geq 1$ (see also~\cite{sakagawa2024maximum} which identifies the size of the maximum of the interface).

\subsection{Sketch of proof} \label{sec:sketchofproof}

The proof of Theorem \ref{thm:qualitative-delocalisation} is treated in Section~\ref{sec:qualitative-delocalisation}. We argue by contradiction, assume the existence of a shift-covariant Gibbs measure and show that it implies the two following contradicting statements. In Lemma~\ref{lem:PhiB_Integrable}, we use the shift-covariance property to prove the upper bound: for any discrete interval $B \subseteq \bbZ$,
\begin{equation*}
\mathbb E\bigl[\mu^\eta \big[ \big|\sum_{x\in B}\sum_{y\notin B}
    |x-y|^{-\alpha}\bigl(\phi(y)-\phi(x)\bigr)\big|\big]\bigr]
    \le
    C \mathbb E \bigl[\mu^\eta[|\phi(0)|]\bigr] \Cap(B).
\end{equation*}
with $\Cap(B):= \sum_{x \in B} \sum_{y \notin B} |x-y|^{-\alpha}$ (see Section~\ref{sec:sec2.5}).
In Lemma~\ref{lem:balance-estimate}, we show that any infinite-volume Gibbs measure satisfies the inequality:
for any discrete interval \(B\subseteq \mathbb Z\),
\begin{equation*}
    \bigl|\lambda \sum_{x \in B} \eta_x+\mu^\eta \big[\sum_{x\in B}\sum_{y\notin B}
    |x-y|^{-\alpha}\bigl(\phi(y)-\phi(x)\bigr) \big] \bigr|
    \le
    \frac12\operatorname{Cap}_\alpha(B)
    \qquad\text{for \(\mathbb P\)-a.e. \(\eta\).}
\end{equation*}
We then observe that the sum $\sum_{x \in B} \eta_x$ is typically of order $|B|^{1/2}$, which, for sufficiently large $|B|$, is larger than $\operatorname{Cap}_\alpha(B)$ (see~\eqref{eq:bound_capacity}). Combining this observation with the two previous inequalities leads to a contradiction, thereby ruling out the existence of a shift-covariant Gibbs measure.

In order to simplify the presentation of the quantitative argument, we will sketch the proof for the ground state of the model in the case $\lambda = 1$, i.e., for the solution of the minimisation problem
\begin{equation} \label{eq:groundstatesketch}
    \min_{\phi \in \Omega_N}  \frac12 \sum_{\{x,y\}\subset\mathbb Z, x\ne y} |x-y|^{-\alpha}(\phi(x)-\phi(y))^2 - \sum_{x\in\Lambda_N} \eta_x \phi(x).
\end{equation}
To further simplify the notation, we will assume that the minimiser of~\eqref{eq:groundstatesketch} is always unique and will denote it by $\phi^\eta_N$ (N.B. This uniqueness property only holds on a set of Lebesgue full measure over the disorder $\eta \in \R^{\Lambda_N}$).

Our strategy relies on comparing the integer-valued ground state of~\eqref{eq:groundstatesketch}, to the real-valued ground state, i.e. the minimizer of
\begin{equation*}
    \min_{u \in \Omega_N^\R}  \frac12 \sum_{\{x,y\}\subset\mathbb Z, x\ne y} |x-y|^{-\alpha}(u(x)-u(y))^2 - \sum_{x\in\Lambda_N} \eta_x u(x).
\end{equation*}
This minimisation problem has a unique solution for every realisation of the disorder $\eta \in \R^{\Lambda_N}$ and the minimiser is given by the formula
\begin{equation} \label{eq:def.groundstatesketch}
    G_{N , \alpha} \eta(x) := \sum_{y \in \Lambda_N}  G_{N , \alpha}(x, y) \eta_y
\end{equation}
where the function $G_{N, \alpha}$ is the discrete fractional Green's function (see Section~\ref{sec:sec2.2} below). It can be equivalently characterised as the unique solution of the linear equation (see Section~\ref{sec:sec2.2} and~\eqref{eq:defoperatuerLNalpha} for the definition of the operator $L_{N , \alpha}$)
\begin{equation} \label{eq:def2GSrealvalued}
    \left\{ 
    \begin{aligned}
    -L_{N , \alpha} (G_{N , \alpha} \eta) = \eta &~\mbox{in}~ \Lambda_N, \\
    G_{N , \alpha} \eta = 0  &~\mbox{in}~ \mathbb{Z} \setminus \Lambda_N.
    \end{aligned} \right.
\end{equation}
The core of the argument is then the proof of the two following results:
\begin{itemize}
    \item One has the asymptotic expansion
    \begin{equation} \label{eq:asympexpGSsketch}
        \mathbb{E} \left[ (G_{N , \alpha} \eta(0))^2 \right] = 
        \left\{ \begin{aligned}
        N^{2 \alpha - 3} \left( \int_{-1}^1 G_{\alpha}(x)^2 \, dx + o_{N \to \infty} (1 ) \right) & \mbox{ for } \alpha \in (3/2, 3), \\
        \frac{N^3}{(\ln N)^2} \left( \int_{-1}^1 G(x)^2 \, dx +  o_{N \to \infty} \left(1\right) \right)  & \mbox{ for } \alpha =3 , \\
         N^3 \left( \int_{-1}^1 G(x)^2 \, dx +  o_{N \to \infty} \left(1\right) \right)  & \mbox{ for } \alpha \in (3 , \infty],
        \end{aligned} \right.
    \end{equation}
   and the central limit theorem, for any $\alpha \in (3/2 , \infty),$
    \begin{equation}  \label{eq:CLTsketch}
         \frac{G_{N , \alpha} \eta(0)}{\sqrt{ \mathbb{E} \left[ (G_{N , \alpha} \eta(0))^2 \right] }} \overset{\mathcal{L}}{\underset{N \to \infty}{\longrightarrow}} \mathcal N(0,1).
    \end{equation}
    \item One has the asymptotic comparison
    \begin{equation} \label{eq:GSclosesketch}
        \mathbb{E} \left[ \left( \phi^\eta_N(0) - G_{N , \alpha} \eta(0) \right)^2 \right] = o_{N \to \infty} \left( \mathbb{E} \left[ (G_{N , \alpha} \eta(0))^2 \right] \right)
    \end{equation}
\end{itemize}
Theorems~\ref{thm:deloc} and~\ref{thm:thm1.4TCL} follow directly from these results.

The proof of the expansion~\eqref{eq:asympexpGSsketch} and the central limit theorem~\eqref{eq:CLTsketch} is presented in Section~\ref{sec:scalinglimrealvaluedGS} (for~\eqref{eq:asympexpGSsketch}) and Section~\ref{sec:secCLT} (for~\eqref{eq:CLTsketch}). The starting point is the identity~\eqref{eq:def.groundstatesketch} which implies
\begin{equation*}
    \mathbb{E} \left[ (G_{N , \alpha} \eta(0))^2 \right] = \sum_{x \in \Lambda_N}  G_{N , \alpha}(0, x)^2.
\end{equation*}
We then study the asymptotic behaviour of the right-hand side above by relying on probabilistic arguments : the map $G_{N , \alpha}$ is the Green's function of a long-range random walk (see Section~\ref{sec:sec2.2}), and its asymptotic behaviour can be studied by using the results on the convergence of the occupation measure for discrete symmetric stable processes of~\cite[Lemma 12.3]{LSSW17} (in the case $\alpha \in (3/2, 3)$), the invariance principle of~\cite[Theorem 1]{sepanski1997some} (for $\alpha = 3$) and the pointwise convergence of Green's function of~\cite[Section 23, Theorem 1]{spitzer1976principles} (for $\alpha \in (3 , \infty]$). The central limit theorem~\eqref{eq:CLTsketch} is a direct consequence of the Lindeberg--Feller theorem (see e.g. \cite[Theorem 3.4.10]{Durrett_2019} or \cite[Theorem 27.2]{Bilingsley_1995}) which generalises the central limit theorem to triangular arrays of random variables, once the following convergence has been proved
\begin{equation*}
        \max_{y \in \Lambda_N}  \frac{G_{N , \alpha}(0, y)^2}{ \sum_{x \in \Lambda_N} G_{N , \alpha}(0, x)^2} \underset{N \to \infty}{\longrightarrow} 0.
\end{equation*}
This convergence is obtained by proving upper bounds on the Green's function (see Proposition~\ref{prop:propgreensfunction}) together with the convergence~\eqref{eq:asympexpGSsketch}.

The proof of~\eqref{eq:GSclosesketch} can be found in Section~\ref{sec:quantdelocheightorigin}. It is a combination of the following results:
\begin{itemize}
\item For any discrete interval $I \subseteq \Lambda_N$, one has the upper bound
\begin{equation} \label{eq:upperboundintervalssketch}
    \left| \sum_{x \in I} L_{N , \alpha} \left( \phi^\eta_N - G_{N , \alpha} \eta \right) (x) \right| \leq \operatorname{Cap}_\alpha(I).
\end{equation}
\item For any function $f : \mathbb{Z} \to \R$ which is equal to $0$ in $\mathbb{Z} \setminus \Lambda_N$ and which satisfies, for any discrete interval $I \subseteq \Lambda_N$,
\begin{equation} \label{eq:upperboundspatialaverageLNf}
 \left| \sum_{x \in I} L_{N , \alpha} f (x) \right| \leq \operatorname{Cap}_\alpha(I),
\end{equation}
one has the upper bound
\begin{equation} \label{eq:RHSdiffgroundstatesketch}
    \left| f (0) \right| \leq C \begin{cases}
        \ln N & \mbox{ for } \alpha \in (1,2), \\
        (\ln N)^2 & \mbox{ for } \alpha =2, \\
        N^{\alpha-2} & \mbox{ for } \alpha \in (2,3), \\
        N/\ln N & \mbox{ for } \alpha =3,\\
        N & \mbox{ for } \alpha >3.
    \end{cases}
\end{equation}
\end{itemize}
The inequality~\eqref{eq:GSclosesketch} follows by combining the two previous items with $ f := \phi^\eta_N - G_{N , \alpha} \eta$ and by noting that the right-hand side of~\eqref{eq:RHSdiffgroundstatesketch} is always smaller than the one of~\eqref{eq:asympexpGSsketch}.

The proof of~\eqref{eq:upperboundintervalssketch} is obtained by testing the functions $\phi_1 := \phi^\eta_N + \mathbf{1}_I$ and $\phi_2 := \phi^\eta_N - \mathbf{1}_I$ in the minimisation problem~\eqref{eq:groundstatesketch} (and by using the identity~\eqref{eq:def2GSrealvalued}). The proof of~\eqref{eq:RHSdiffgroundstatesketch} is obtained by writing (see~\eqref{eq:def2GSrealvaluedsec2} below)
\begin{equation*}
    f(0) = \sum_{x \in \Lambda_N} G_{N , \alpha} (0, x) (- L_{N , \alpha}) f (x),
\end{equation*}
and by using the upper bound~\eqref{eq:upperboundspatialaverageLNf} together with the properties of the Green's function $G_{N , \alpha}$ listed in Proposition~\ref{prop:propgreensfunction}.

The extension of these results from the ground state to the Gibbs measure~\eqref{eq:defGibbsmeasure} requires an additional input: instead of relying on the minimisation problem~\eqref{eq:groundstatesketch}, we make use of the Boltzmann--Gibbs variational principle (see Section~\ref{sec:sectionBoltzmannGibbs}) which identifies the Gibbs measure~\eqref{eq:defGibbsmeasure} as the solution of a minimisation problem (over the probability measures on~$\Omega_N$) involving an energy term (given by the Hamiltonian $\mathcal H_{N, \alpha,\lambda}^\eta$) and an entropy term. The general strategy of the proof is then similar to the one sketched above for the ground state with additional technicalities due to the incorporation of the entropic term in the minimisation principle.

\subsection{Convention for constants} Throughout this article, the symbols $C$ and $c$ denote positive constants which may vary from line to line, with $C$ increasing and $c$ decreasing. Except if explicitly stated otherwise, these constants may depend on the decay exponent $\alpha$, the inverse temperature $\beta$ and the disorder strength $\lambda$.

\section{Notation and preliminary results}

\subsection{Notation}
For $k \in \mathbb{Z}$ (or $\mathbb{R}$), we set $|k|_+ := 2+ |k|$.

When $\alpha,\lambda$ and $\beta$ are fixed (i.e. most of the time in the rest of the paper), we abbreviate
\[
    \Gib:=\Gibbs, \qquad \E:=\EE  \qquad \mbox{and}\qquad \mathcal{H}^\eta_N:=\mathcal{H}^\eta_{N, \alpha, \lambda}.
\]

For $B \subset \Lambda_N$, we denote by $\mathbf{1}_B$ the element of $\Omega_N$ whose coordinates are $0$ everywhere but in $B$, where they equal $1$; when $B=\{x\}$, we write $\delta_x:=\mathbf{1}_{\{x\}}$.

\subsection{The Dirichlet Green's function} \label{sec:sec2.2}

Let
\[
    \Omega_N^{\mathbb R}
    :=
    \{u:\mathbb Z\to\mathbb R:\ u(x)=0 \text{ for all }x\notin\Lambda_N\}
    \simeq \mathbb R^{\Lambda_N},
\]
and let $L_{N, \alpha}$ be the operator on $\Omega_N^\R$ defined by
\begin{equation} \label{eq:defoperatuerLNalpha}
    L_{N, \alpha} f(x)=\sum_{y \neq x}\dfrac{f(y)-f(x)}{|y-x|^\alpha}, \qquad \forall f\in \Omega_N^{\mathbb R},  \quad \forall x \in \Lambda_N.
\end{equation}
It is straightforward to check that $\mathcal{H}_{N, \alpha}^0(\phi)= - \frac12 \langle \phi, L_{N, \alpha} \phi \rangle$ for all $\phi \in \Omega_N$, where $\langle \cdot,\cdot\rangle$ denotes the standard scalar product on $\ell^2(\Lambda_N)$.
This observation implies that $- L_{N, \alpha}$ is a positive definite symmetric operator, so it admits a (symmetric) inverse denoted by $G_{N, \alpha}$.
For any function $h : \Lambda_N \to \R$, let us denote by
\begin{equation*}
   \forall x \in \Lambda_N, ~ G_{N , \alpha}h(x) = \sum_{x \in \Lambda_N}  G_{N , \alpha}(x , y) h(y),
\end{equation*}
so that the function $G_{N , \alpha}h$ is the unique solution of the equation
    \begin{equation} \label{eq:def2GSrealvaluedsec2}
   \left\{ 
    \begin{aligned}
    -L_{N , \alpha} f = h &~\mbox{in}~ \Lambda_N, \\
    f = 0  &~\mbox{in}~ \mathbb{Z} \setminus \Lambda_N.
    \end{aligned} \right.
\end{equation}
It is well-known that this inverse is the Green's function of the continuous-time random walk whose generator is $L_{N, \alpha}$.
Thus, we have
\begin{equation*}
    G_{N, \alpha}(x,y) = \langle \delta_x,(-L_{N, \alpha})^{-1} \delta_y \rangle =  \int_0^{+\infty} p_t^{\Lambda_N}(x,y) \mathrm{d}t \qquad \text{with }p_t^{\Lambda_N}(x,y)=\langle \delta_x,e^{t L_{N, \alpha}} \delta_y \rangle.
\end{equation*}
Equivalently, $G_{N, \alpha}$ is (up to some multiplicative factor) the Green's function of the discrete-time long-range random walk $(X_k)_{k \geq 0}$ with transition kernel $q_\alpha$ given by $$q_\alpha(0)=0 \qquad \mbox{and}\qquad q_\alpha(n)=\dfrac{|n|^{-\alpha}}{c_\alpha} \, \forall n\neq 0,$$
where $c_\alpha=2\sum_{n \geq 1} n^{-\alpha}$ is the total jump rate of the continuous-time random walk.
More precisely,
$$G_{N, \alpha} (x,y)= c_\alpha^{-1} \mathbf E\left[ \sum_{k=0}^{\tau_N-1} \mathbf{1}_{\{ X_k=y \}} \mid X_0=x\right]$$
with $\tau_N := \inf\{k\ge0:\ X_k\notin\Lambda_N\}$ and $\mathbf E$ denoting the expectation with respect to the random walk.

Either of these representations (continuous or discrete) may be used interchangeably, depending on which is more suitable.
We list in the following proposition several properties of this Green's function that we will use extensively in the rest of the paper. Note that none of these properties is new, but we provide proofs (or precise references) in Appendix~\ref{appendix}.

\begin{proposition}[Properties of the Green's function]\label{prop:propgreensfunction}
The following properties hold:
\begin{enumerate}
    \item (Vanishing) For all $x \in \mathbb Z \setminus \Lambda_N$, $G_{N,\alpha}(0,x)=0$
    \item (Symmetry and monotonicity) The function $x \mapsto G_{N, \alpha}(0,x)$ is even and non-increasing on $\N = \{0 , 1 , 2 , \ldots\}$.
    \item (Asymptotic upper bounds) There exists some constant $C$ depending only on $\alpha$ such that for all $N \in \N$ and all $x \in \Lambda_N$
    \begin{equation*}
        G_{N, \alpha}(0,x) \leq C \begin{cases}
            |x|_+^{\alpha-2} & \mbox{ for }\alpha \in (1,2), \\
            1+\ln(N/|x|_+) & \mbox{ for }\alpha =2, \\
            N^{\alpha-2} & \mbox{ for }\alpha \in (2,3), \\
            N/\ln N & \mbox{ for }\alpha =3, \\
            N & \mbox{ for }\alpha >3. \\
        \end{cases}
    \end{equation*}
\end{enumerate}
\end{proposition}

\begin{remark}
    We mention that, although this will not be done in this article, it would be possible to prove lower bounds on the Green's function matching the upper bound stated above up to a multiplicative constant for the vertices in the bulk of the discrete interval $\{ -N , \ldots, N\}$. 
\end{remark}

\subsection{The real-valued Gaussian model} \label{sec:sectionrealvaluedgaussian}
This section is devoted to the real-valued Gaussian model defined in~\eqref{eq:defreal-valuedmodel}. As mentioned above, this model is a multivariate normal distribution with mean and covariance matrix explicitly given by
\begin{equation} \label{eq:Gaussianidentity}
    \mu_{N, \alpha, \lambda, \beta}^{\eta, \R}
    =
    \mathcal N\left(
        \lambda G_{N,\alpha}\eta,\,
        \frac1\beta G_{N,\alpha}
    \right) ~~\mbox{with}~~ (G_{N,\alpha}\eta)(x)
    :=
    \sum_{y\in\Lambda_N}G_{N,\alpha}(x,y)\eta_y.
\end{equation}
When $\beta$ goes to infinity, the measure $\mu_{N, \alpha, \lambda, \beta}^{\eta, \R}$ concentrates around its mean $\lambda G_{N, \alpha} \eta$; this corresponds to the real-valued ground state, which is the configuration in $\Omega^\R_N$ that minimises the Hamiltonian.
Since the random variables $(\eta_x)_{x \in \mathbb Z}$ are assumed to be independent, centered and of variance $1$, we have
\begin{equation} \label{eq:sumandvarianceiid}
\mathbb{E}\left[ G_{N , \alpha}\eta(0) \right] = 0 ~~\mbox{and}~~\mathbb{E}\left[  G_{N , \alpha}\eta(0)^2 \right] = \sum_{x \in \Lambda_N}  G_{N , \alpha}(0, x)^2.
\end{equation}
These observations combined with the convergences~\eqref{eq:asympexpGSsketch} and~\eqref{eq:CLTsketch} and the upper bounds stated in Proposition~\ref{prop:propgreensfunction} are enough to prove Theorems~\ref{thm:deloc} and~\ref{thm:thm1.4TCL} in the case of the real-valued model.

The discrete model does not satisfy such an exact Gaussian identity~\eqref{eq:Gaussianidentity}. This makes the analysis more complicated and a substantial part of the proof is devoted to showing that
$$\bbE\left[ E_N^\eta\left[(\phi(0) - \lambda G_{N , \alpha}\eta(0))^2\right]\right] \ll \sum_{x \in \Lambda_N}  G_{N , \alpha}(0, x)^2$$
in order to show that the ground state  $\lambda G_{N , \alpha}\eta$ and the random variable $\phi(0)$ (with $\phi$ sampled according to the annealed distribution $ \mathbb P_{N,\alpha,\lambda, \beta}^{\mathrm{ann}}$) have the same asymptotic behaviour.

\subsection{Entropy and Boltzmann--Gibbs variational principle} \label{sec:sectionBoltzmannGibbs}
As already mentioned in the introduction, the Boltzmann--Gibbs variational principle will be a crucial tool of the proofs. Before stating the principle, which is nothing but the mathematical version of the $2^{\text{nd}}$ law of thermodynamics,
let us recall a few facts about the entropy of random variables.\\

Let $\nu$ be a probability distribution on a discrete set $\Omega$. The entropy of $\nu$, or equivalently of a random variable of law $\nu$, is defined by
\begin{equation*}
    \Ent(\nu):= - \sum_{\omega \in \Omega} \nu(\omega) \ln \nu(\omega)\quad \text{(with the convention } 0 \ln 0 = 0).
\end{equation*}
(N.B. here, we choose to define the entropy as a positive quantity).

Now, recall the definition of the conditional entropy. For $(X,Y)$ a couple of random variables, for any $y$ in the support of $Y$, $\Ent(X \mid Y=y)$ denotes the entropy of the conditional distribution of $X$ given $Y=y$, that is
$$\Ent(X \mid Y=y) :=-\sum_x \bfP(X=x \mid Y=y) \ln \bfP(X=x \mid Y=y)$$
($\bfP$ being the law of the couple).
The conditional entropy, of $X$ given $Y$, is then defined as the expectation over all $y$ of the above quantity:
$$\Ent(X \mid Y):= \sum_y \bfP(Y=y) \Ent(X \mid Y=y) =- \sum_{x,y} \bfP(X=x,Y=y)\ln\bfP(X=x \mid Y=y).$$
A simple computation shows that 
\begin{equation*}
    \Ent(X)-\Ent(Y) = \Ent(X \mid Y) - \Ent(Y\mid X),
\end{equation*}
so that the difference of entropy between two distributions is given by the difference of conditional entropy, for any coupling of the two laws.
It is also easy to see that if $Y$ is a measurable function of $X$, then $\Ent(Y \mid X)=0$, so $\Ent(X) \geq \Ent(Y)$. The application of a measurable function can therefore only decrease the entropy.\\

We now state the Boltzmann--Gibbs variational principle. Although it applies to a wide class of statistical mechanics models, here we only state it in the context of our Gibbs measure $\Gib$ on $\Omega_N$.
In finite volume, it says that the Gibbs measure is the (unique) maximiser amongst all probability distributions on $\Omega_N$ of the quantity $\Ent(\nu)- \beta \nu[ \mathcal{H}_N^\eta]$, where $\nu[ \cdot]$ denotes the expectation with respect to $\nu$. In a practical way, we have for every probability measure $\nu$ on $\Omega_N$,
\begin{equation}\label{eq:BGVP}\tag{BGVP}
    \Ent(\nu)-\beta \nu[\mathcal{H}_N^\eta] \leq \Ent(\Gib)-\beta \E[\mathcal{H}_N^\eta].
\end{equation}

\subsection{Capacity of a subset of $\mathbb Z$} \label{sec:sec2.5}
In this short section, we introduce the capacity of a subset of $\mathbb Z$, defined as the total interaction between the subset and its complement.
For every finite subset $B \Subset \mathbb Z$, let
\begin{equation*}
    \Cap(B):= \sum_{x \in B} \sum_{y \notin B} |x-y|^{-\alpha}.
\end{equation*}
For future purpose, observe that for all $B \subset \Lambda_N$, we can rewrite this as
$$\Cap(B) = 2\mathcal{H}^0_{N, \alpha}(\mathbf{1}_B) = \langle \mathbf{1}_B, - L_{N, \alpha}\mathbf{1}_B \rangle.$$
In potential theory, the capacity of a set $A$ relative to a larger set $B$ is classically defined as the minimal Dirichlet energy of a function that equals 1 on $A$ and vanishes outside $B$. The capacity of a set $B$ is then often defined as the capacity of $B$ relative to the full space, but here we consider the capacity of $B$ relative to itself.
Similarly, in a more probabilistic setting, the capacity of $A$ relative to $B$ is the sum over $x \in A$ of the probability that a random walk starting at $x$ leaves $B$ before returning in $A$; our definition of $\Cap(B)$ matches up to a multiplicative factor $c_\alpha$ this probabilistic interpretation of the capacity of $B$ relative to itself.

In several places in the proofs, we will need the following upper bound on the capacity of an interval: there exists a constant $C$ depending only on $\alpha$ such that for any interval $I \Subset \mathbb Z$,
\begin{equation}\label{eq:bound_capacity}
    \Cap(I) \leq C \begin{cases}
        |I|^{2- \alpha} & \text{ if } \alpha \in(1,2)\\
        \ln |I|_+ & \text{ if } \alpha=2 \\
        1 & \text{ if } \alpha>2
    \end{cases}
\end{equation}
where $|I|$ denotes the length of $|I|$. The proof is a simple computation. Indeed, by translation-invariance, we can identify $I$ with $\{1, \dots, |I|\}$ and we have
$$\Cap(\{1, \dots, |I|\})=2\sum_{x=1}^{|I|} \sum_{n \geq x} n^{-\alpha} \leq C \sum_{x=1}^{|I|} x^{1-\alpha}$$
which yields \eqref{eq:bound_capacity}.

\section{Qualitative delocalisation for $\alpha>3/2$}
\label{sec:qualitative-delocalisation}

In this section we prove Theorem~\ref{thm:qualitative-delocalisation}. We fix throughout the section
\[
    \alpha\in(3/2,\infty],\qquad \lambda>0,\qquad \beta>0,
\]
and recall the definition of the set of admissible height functions $\Omega_\alpha$ introduced in~\eqref{def:defOmega}.
We first introduce two definitions: the finite-volume specification (following~\cite{georgii}) and the shift-covariant infinite-volume Gibbs measure.

\begin{definition}[Finite-volume specification] \label{def:3.1}
For $\zeta \in \Omega_\alpha$, the finite-volume Hamiltonian in \(\Lambda \subseteq \bbZ \) with boundary condition \(\zeta\) is
\begin{multline*}
    \mathcal H^\eta_{\Lambda,\alpha,\lambda}(\xi\mid\zeta)
    :=
    \frac12
    \sum_{\substack{\{x,y\}\subset\Lambda\\x\neq y}}
    |x - y|^{-\alpha}\bigl(\xi(x)-\xi(y)\bigr)^2
    \\
    +
    \frac12
    \sum_{x\in\Lambda}\sum_{y\notin\Lambda}
    |x - y|^{-\alpha} \left[ \bigl(\xi(x)-\zeta(y)\bigr)^2 - \zeta(y)^2 \right]
    -
    \lambda\sum_{x\in\Lambda}\eta_x\xi(x).
\end{multline*}
We then define the finite-volume specification to be the probability kernel $(\gamma_\Lambda^\eta)_\Lambda$ with
\[
    \gamma^\eta_{\Lambda,\alpha,\lambda,\beta}(\xi\mid\zeta)
    :=
    \frac{1}{Z^\eta_{\Lambda,\alpha,\lambda,\beta}(\zeta)}
    \exp\bigl(-\beta \mathcal H^\eta_{\Lambda,\alpha,\lambda}(\xi\mid\zeta)\bigr),
    \qquad \xi\in\mathbb Z^\Lambda,
\]
where
\[
    Z^\eta_{\Lambda,\alpha,\lambda,\beta}(\zeta)
    :=
    \sum_{\xi\in\mathbb Z^\Lambda}
    \exp\bigl(-\beta \mathcal H^\eta_{\Lambda,\alpha,\lambda}(\xi\mid\zeta)\bigr).
\]
\end{definition}

\begin{remark} \label{remark3.2}
    In the second sum of the definition of the Hamiltonian $\mathcal H^\eta_{\Lambda,\alpha,\lambda}(\xi\mid\zeta)$, the term $\zeta(y)^2$ was subtracted to ensure that the infinite sum is well-defined (i.e., converges) for any value of $\zeta \in \Omega_\alpha$ (N.B. we use here the assumption that $\zeta \in \Omega_\alpha$ as the sum does not converge for any function $\zeta : \bbZ \to \bbZ$).
\end{remark}

We next introduce the definition of shift-covariant infinite-volume Gibbs measure.

\begin{definition}[Shift-covariant infinite-volume Gibbs measure] \label{def:def3.1}
A measurable map
\[
    \eta\mapsto\mu^\eta\in\mathcal P(\Omega_\alpha)
\]
is called an infinite-volume shift-covariant random Gibbs measure if it satisfies the two following properties:
\begin{itemize}
\item \emph{The Dobrushin-Lanford-Ruelle (DLR) condition:} for \(\mathbb P\)-a.e. \(\eta\), for every \(\Lambda\subseteq\mathbb Z\), and for every bounded local measurable function \(f\),
\[
    \mu^\eta[f]
    =
    \int_{\Omega_\alpha}
    \gamma^\eta_{\Lambda,\alpha,\lambda,\beta}[f\mid\phi]\,\mu^\eta(d\phi).
    \tag{DLR}
\]
Here \(\gamma^\eta_{\Lambda,\alpha,\lambda,\beta}[f\mid\phi]\) is the expectation of $f$ with respect to $\gamma^\eta_{\Lambda,\alpha,\lambda,\beta}(\cdot \mid\phi)$.
\item \emph{Shift-covariance:} for every \(a\in\mathbb Z\),
\[
    \mu^{\theta_a\eta}=\mu^\eta\circ\tau_a^{-1}
    \qquad\text{for \(\mathbb P\)-a.e. \(\eta\)}
\]
with
\[
    (\theta_a\eta)_x:=\eta_{x+a},
    \qquad
    (\tau_a\phi)(x):=\phi(x+a).
\]
\end{itemize}
\end{definition}
For \(B\Subset\mathbb Z\), define
\[
    S_B(\eta):=\sum_{x\in B}\eta_x
\]
and, for any $\phi \in \Omega_\alpha$,
\[
    \Phi_B(\phi):=
    \sum_{x\in B}\sum_{y\notin B}
    |x - y|^{-\alpha}\bigl(\phi(y)-\phi(x)\bigr).
\]
\begin{lemma}[Integrability of \(\Phi_B\)]
\label{lem:PhiB_Integrable}
Let \(\eta\mapsto\mu^\eta\) be shift-covariant and assume that
\[
    M:=\mathbb E\bigl[\mu^\eta[|\phi(0)|]\bigr]<\infty.
\]
Then, for every \(B\Subset\mathbb Z\),
\[
    \mathbb E\bigl[\mu^\eta[|\Phi_B|]\bigr]
    \le
    2M\operatorname{Cap}_\alpha(B).
\]
In particular, \(\Phi_B\) is integrable under the annealed measure \(\mathbb P(d\eta)\mu^\eta(d\phi)\), and \(\mu^\eta[\Phi_B]\) is well-defined for \(\mathbb P\)-a.e. \(\eta\).
\end{lemma}

\begin{proof}
By shift-covariance and translation-invariance of the disorder law \(\mathbb P\), for every \(z\in\mathbb Z\),
\[
    \mathbb E\bigl[\mu^\eta[|\phi(z)|]\bigr]
    =
    \mathbb E\bigl[\mu^\eta[|\phi(0)|]\bigr]
    =
    M.
\]
Indeed, using shift-covariance with \(a=z\),
\[
    \mu^\eta[|\phi(z)|]
    =
    \mu^{\theta_z\eta}[|\phi(0)|],
\]
and the law of \(\theta_z\eta\) under \(\mathbb P\) is again \(\mathbb P\).

Therefore, by Tonelli's theorem and the triangle inequality,
\begin{align*}
    \mathbb E\bigl[\mu^\eta[|\Phi_B|]\bigr]
    &\le
    \sum_{x\in B}\sum_{y\notin B}
    |x - y|^{-\alpha}
    \mathbb E\bigl[\mu^\eta[|\phi(y)-\phi(x)|]\bigr]
    \\
    &\le
    \sum_{x\in B}\sum_{y\notin B}
    |x - y|^{-\alpha}
    \left(
        \mathbb E\bigl[\mu^\eta[|\phi(y)|]\bigr]
        +
        \mathbb E\bigl[\mu^\eta[|\phi(x)|]\bigr]
    \right)
    \\
    &=
    2M\operatorname{Cap}_\alpha(B).
\end{align*}
This proves the lemma.
\end{proof}

\begin{lemma}
\label{lem:balance-estimate}
Let \(\eta\mapsto\mu^\eta\) be an infinite-volume random Gibbs measure which is shift-covariant and satisfies
\[
    \mathbb E\bigl[\mu^\eta[|\phi(0)|]\bigr]<\infty.
\]
Then, for every \(B\Subset\mathbb Z\),
\begin{equation*}
    \bigl|\lambda S_B(\eta)+\mu^\eta[\Phi_B]\bigr|
    \le
    \frac12\operatorname{Cap}_\alpha(B)
    \qquad\text{for \(\mathbb P\)-a.e. \(\eta\).}
\end{equation*}
\end{lemma}

\begin{proof}
Fix \(B\Subset\mathbb Z\), a disorder realisation \(\eta\), and an exterior configuration \(\zeta\in \Omega_\alpha\). Let \(\xi\in\mathbb Z^B\), and write \(\xi\zeta\) for the full configuration equal to \(\xi\) on \(B\) and to \(\zeta\) on \(B^c\).

Since adding a constant to every coordinate in \(B\) leaves the internal quadratic energy unchanged, we have
\begin{align}
    \mathcal H^\eta_{B,\alpha,\lambda}(\xi+\mathbf{1}_B\mid\zeta)
    -
    \mathcal H^\eta_{B,\alpha,\lambda}(\xi\mid\zeta)
    &=
    -\Phi_B(\xi\zeta)
    +
    \frac12\operatorname{Cap}_\alpha(B)
    -
    \lambda S_B(\eta),
    \label{eq:block-shift-plus}
    \\
    \mathcal H^\eta_{B,\alpha,\lambda}(\xi-\mathbf{1}_B\mid\zeta)
    -
    \mathcal H^\eta_{B,\alpha,\lambda}(\xi\mid\zeta)
    &=
    \Phi_B(\xi\zeta)
    +
    \frac12\operatorname{Cap}_\alpha(B)
    +
    \lambda S_B(\eta).
    \label{eq:block-shift-minus}
\end{align}
Indeed, for the \(+\mathbf{1}_B\) shift, the boundary energy changes by
\begin{align*}
    &\frac12
    \sum_{x\in B}\sum_{y\notin B}
    |x - y|^{-\alpha}
    \left(
        (\xi(x)+1-\zeta(y))^2
        -
        (\xi(x)-\zeta(y))^2
    \right)
    \\
    &\qquad =
    \sum_{x\in B}\sum_{y\notin B}
    |x - y|^{-\alpha}\bigl(\xi(x)-\zeta(y)\bigr)
    +
    \frac12\operatorname{Cap}_\alpha(B)
    \\
    &\qquad =
    -\Phi_B(\xi\zeta)
    +
    \frac12\operatorname{Cap}_\alpha(B),
\end{align*}
and the disorder contribution changes by \(-\lambda S_B(\eta)\). This proves \eqref{eq:block-shift-plus}; the proof of \eqref{eq:block-shift-minus} is identical.

Since \(\xi\mapsto \xi+\mathbf{1}_B\) is a bijection of \(\mathbb Z^B\),
\[
    \gamma^\eta_{B,\alpha,\lambda,\beta}
    \left[
        \exp\left(
            -\beta\left(
                \mathcal H^\eta_{B,\alpha,\lambda}(\xi+\mathbf{1}_B\mid\zeta)
                -
                \mathcal H^\eta_{B,\alpha,\lambda}(\xi\mid\zeta)
            \right)
        \right)
        \,\middle|\,\zeta
    \right]
    =
    1.
\]
By Jensen's inequality,
\[
    \gamma^\eta_{B,\alpha,\lambda,\beta}
    \left[
        \mathcal H^\eta_{B,\alpha,\lambda}(\xi+\mathbf{1}_B\mid\zeta)
        -
        \mathcal H^\eta_{B,\alpha,\lambda}(\xi\mid\zeta)
        \,\middle|\,\zeta
    \right]
    \ge 0.
\]
Using \eqref{eq:block-shift-plus}, this gives
\[
    \lambda S_B(\eta)
    +
    \gamma^\eta_{B,\alpha,\lambda,\beta}[\Phi_B\mid\zeta]
    \le
    \frac12\operatorname{Cap}_\alpha(B).
\]
Repeating the same argument with the bijection \(\xi\mapsto\xi-\mathbf{1}_B\), and using \eqref{eq:block-shift-minus}, gives
\[
    \lambda S_B(\eta)
    +
    \gamma^\eta_{B,\alpha,\lambda,\beta}[\Phi_B\mid\zeta]
    \ge
    -\frac12\operatorname{Cap}_\alpha(B).
\]
Thus
\begin{equation}\label{eq:block-shift-conditional}
    \left|
        \lambda S_B(\eta)
        +
        \gamma^\eta_{B,\alpha,\lambda,\beta}[\Phi_B\mid\zeta]
    \right|
    \le
    \frac12\operatorname{Cap}_\alpha(B).
\end{equation}
By Lemma~\ref{lem:PhiB_Integrable},
$\Phi_B$ is integrable under $\mathbb P(d\eta)\mu^\eta(\cdot)$, hence for
$\mathbb P$-a.e. $\eta$ it is $\mu^\eta$-integrable. The DLR identity yields
\[
    \int_{\mathbb{Z}^{\mathbb{Z}}}
    \gamma^\eta_{B,\alpha,\lambda,\beta}[\Phi_B\mid\zeta]
    \,\mu^\eta(d\zeta)
    =
    \mu^\eta[\Phi_B]
    \qquad\text{for \(\mathbb P\)-a.e. \(\eta\)}.
\]
Integrating \eqref{eq:block-shift-conditional} with respect to $\mu^\eta$ gives
\[
    \bigl|\lambda S_B(\eta)+\mu^\eta[\Phi_B]\bigr|
    \le
    \frac12\operatorname{Cap}_\alpha(B),
\]
which proves the lemma.
\end{proof}

\begin{proof}[Proof of Theorem~\ref{thm:qualitative-delocalisation}]
Assume, for contradiction, that there exists an infinite-volume random Gibbs measure \(\eta\mapsto\mu^\eta\) satisfying shift-covariance and
\[
    M:=\mathbb E\bigl[\mu^\eta[|\phi(0)|]\bigr]<\infty.
\]
For \(N\ge1\), set
\[
    I_N:=\{1,\ldots,N\}.
\]
By Lemma~\ref{lem:balance-estimate},
\[
    \bigl|\lambda S_{I_N}(\eta)+\mu^\eta[\Phi_{I_N}]\bigr|
    \le
    \frac12\operatorname{Cap}_\alpha(I_N)
    \qquad\text{for \(\mathbb P\)-a.e. \(\eta\).}
\]
Hence
\begin{align*}
    \lambda\mathbb E[|S_{I_N}|]
    &\le
    \mathbb E\left[
        \bigl|
            \lambda S_{I_N}+\mu^\eta[\Phi_{I_N}]
        \bigr|
    \right]
    +
    \mathbb E\left[
        \bigl|\mu^\eta[\Phi_{I_N}]\bigr|
    \right]
    \\
    &\le
    \frac12\operatorname{Cap}_\alpha(I_N)
    +
    \mathbb E\bigl[\mu^\eta[|\Phi_{I_N}|]\bigr]
    \\
    &\le
    \left(\frac12+2M\right)\operatorname{Cap}_\alpha(I_N),
\end{align*}
where the last inequality follows from Lemma~\ref{lem:PhiB_Integrable}.

For every \(\alpha\in(3/2,\infty]\), the capacity estimates \eqref{eq:bound_capacity} yield
\begin{equation}\label{eq:esperance_S_petito}
    \lambda \bbE[|S_{I_N}|]=o(\sqrt N).
\end{equation}
On the other hand, since the disorder variables are i.i.d. with expectation $0$ and variance $1$, the central limit theorem gives
\[
    \frac{S_{I_N}}{\sqrt N}
    \xrightarrow[N\to\infty]{\mathcal{L}}
    \mathcal N(0,1).
\]
Moreover,
\[
    \mathbb E\left[\left(\frac{S_{I_N}}{\sqrt N}\right)^2\right]=1
    \qquad\text{for every }N\ge1.
\]
Thus the family \(\{|S_{I_N}|/\sqrt N:N\ge1\}\) is uniformly integrable, and consequently
\begin{equation*}
    \frac{1}{\sqrt N}\mathbb E[|S_{I_N}|]
    \longrightarrow
    \mathbb E[|Z|]
    =
    \sqrt{\frac2\pi},
    \qquad Z\sim\mathcal N(0,1).
\end{equation*}
This convergence contradicts \eqref{eq:esperance_S_petito}, which completes the proof.
\end{proof}

\section{Scaling limit of the real-valued ground state} \label{sec:scalinglimrealvaluedGS}
In this section, we identify the typical height and the scaling limit of the real-valued ground state at the origin.
Recall from \eqref{eq:sumandvarianceiid} that this quantity is of order $\lambda$ times the $\mathbb{L}^2$-norm of the discrete fractional killed Green's function.
We prove the convergence of this (suitably rescaled) $\mathbb{L}^2$-norm toward its continuous counterpart.

\begin{proposition}[Asymptotics for the $\mathbb{L}^2$-norm of the discrete Green's function] \label{prop:convergencesquareGreens}
The following results hold:
\begin{itemize}
\item For any $\alpha \in (3/2 , 3),$ one has the convergence
\begin{equation*}
    \frac{1}{N^{2\alpha -3}}\sum_{x\in \Lambda_N} G_{N , \alpha}(0,x)^2 \underset{N \to \infty}{\longrightarrow} \int_{-1}^{1} G_{\alpha}(x)^2 \, dx.
\end{equation*}
\item For $\alpha = 3$, one has the convergence
\begin{equation*}
    \frac{(\ln N)^2}{N^{3}} \sum_{x\in \Lambda_N} G_{N , 3}(0,x)^2 \underset{N \to \infty}{\longrightarrow} \int_{-1}^{1} G(x)^2 \, dx.
\end{equation*}
\item For any $\alpha \in (3 , \infty]$, one has the convergence
\begin{equation*}
    \frac{1}{N^{3}} \sum_{x\in \Lambda_N} G_{N , \alpha}(0,x)^2 \underset{N \to \infty}{\longrightarrow} \frac{1}{\sigma_\alpha^2} \int_{-1}^{1} G(x)^2 \, dx ~~\mbox{with} ~~\sigma_\alpha := \sum_{k =1}^\infty \frac{1}{|k|^{\alpha - 2}}.
\end{equation*}
\end{itemize}
\end{proposition}

In other words, Proposition \ref{prop:convergencesquareGreens} states that the first three items of Theorem \ref{thm:deloc} hold for the ground state of the real-valued model.\\

The rest of this section is organised as follows. Section~\ref{section5.1} proves the weak convergence of the discrete Green's functions toward the continuous ones, Section~\ref{section5.2} provides estimates on its Fourier coefficients and
Section~\ref{section5.3} contains the proof of Proposition~\ref{prop:convergencesquareGreens} relying on Fourier analysis and the convergence established in Section~\ref{section5.1}.

For clarity of the presentation, we decided to write a proof covering all the possible cases for the decay exponent $\alpha \in (3/2, \infty]$, but note that, in the regime $\alpha > 3$, the result directly follows from~\cite[Section 23, Theorem 1]{spitzer1976principles} (see~\eqref{eq:unifconvgreensalpha>3} below).

\subsection{Preliminaries: Convergence of the occupation measure} \label{section5.1}

In this section, we state the weak convergence of the discrete (fractional) Green's functions toward their continuous counterpart. In the case of the decay exponent $\alpha < 3$, the result is a direct consequence of the convergence of the occupation measure for discrete symmetric stable processes from~\cite[Lemma 12.3]{LSSW17}.

\begin{proposition}[Convergence of the occupation measure] \label{prop:convergenceoccupationmeasure}
For any continuous function $h \in C([-1,1])$, one has the convergences
\begin{itemize}
\item For any $\alpha \in (3/2 , 3),$
\begin{equation*}
    \frac{1}{N^{\alpha-1}}\sum_{x \in \Lambda_N} h \left( \frac{x}{N} \right) G_{N , \alpha} (0,x) \underset{N \to \infty}{\longrightarrow} \int_{-1}^1  h \left( x \right) G_{\alpha} (x) \, dx 
\end{equation*}
\item For $\alpha = 3,$
\begin{equation*}
    \frac{\ln N}{N^{2}}\sum_{x \in \Lambda_N} h \left( \frac{x}{N} \right) G_{N , 3} (0,x) \underset{N \to \infty}{\longrightarrow} \int_{-1}^1  h \left( x \right) G (x) \, dx 
\end{equation*}
\item For any $\alpha \in (3 , \infty],$
\begin{equation*}
    \frac{1}{N^{2}}\sum_{x \in \Lambda_N} h \left( \frac{x}{N} \right) G_{N , \alpha} (0,x) \underset{N \to \infty}{\longrightarrow}\frac{1}{\sigma_\alpha} \int_{-1}^1  h \left( x \right) G (x) \, dx 
\end{equation*}
\end{itemize}
\end{proposition}

\begin{proof}
    In the case of the decay exponent $\alpha \in (3/2 , 3)$, the result is the one of~\cite[Lemma 12.3]{LSSW17} (N.B. The convergence proved there holds for any decay exponent $\alpha \in (1 , 3)$).\\

    In the case of the decay exponent $\alpha = 3$, we may first assume, without loss of generality, that $h$ is a Lipschitz function. Let $(Y_n)_{n \in \N}$ be a collection of i.i.d. random variables with distribution given by (see Section~\ref{sec:sec2.2} for the notation)
    \begin{equation*}
        \forall k \in \mathbb{Z} \setminus \{ 0 \}, ~~ \mathbb{P} \left[ Y_n =  k \right] = q_3(k).
    \end{equation*}
    We then denote $X_n := \sum_{k = 1}^n Y_k$ the random walk and let $(X_t)_{t \geq 0}$ be its linear interpolation. The results of Sepanski~\cite[Theorem 1]{sepanski1997some}, which extends Donsker's Theorem to random variables with infinite variance but which are in the basin of attraction of a normal distribution, such as the random variables $(Y_n)_{n \in \N}$ above, implies that
    \begin{equation} \label{eq:invarianceprinciplecriticalcase}
        \left( \frac{X_{N^2 t/ \ln N}}{N} \right)_{t \geq 0} \overset{\mathcal{L}}{\underset{N \to \infty}{\longrightarrow}} \left(\sqrt{\frac{2}{c_3}}B_t\right)_{t \geq 0} ~~\mbox{in}~~C([0 , \infty), \mathbb{R}),
    \end{equation}
    where $(B_t)_{t \geq 0}$ is a standard Brownian motion. 
    
    For $N \in \N$, let us denote by $X^N_t :=  \frac1NX_{N^2 t/ \ln N} $ and let $T^N$ be the stopping time
    \begin{equation*}
        T^N := \inf \left\{ k \in \frac{\ln N}{N^2} \mathbb{N} \, : \, \left| X^N_k \right| \geq 1 \right\}.
    \end{equation*}
    By the definition of the Green's function $G_{N , 3}$, we have
    \begin{equation} \label{eq:identititygreensrandomwalk}
        \frac{\ln N}{N^{2}}\sum_{x \in \Lambda_N} h \left( \frac{x}{N} \right) G_{N , 3} (0,x) = c_3^{-1} \mathbb{E} \left[ \frac{\ln N}{N^2} \sum_{k \in \frac{\ln N}{N^2} \mathbb{N} } h\left(X^N_k\right) \mathbf{1}_{\{ k \leq T^N\}}\right].
    \end{equation}
    By Skorokhod’s representation theorem~\cite[Theorem 6.7]{billingsley2013convergence}, there exists a coupling between the stochastic processes $(X^{N})_{N \in \N}$ such that, for any $T \in (1 , \infty)$,
    \begin{equation*}
        \sup_{t \in [0 , T]} \left|X^N_t - \sqrt{\frac{2}{c_3}} B_t \right| \overset{a.s.}{\underset{N \to \infty}{\longrightarrow}} 0.
    \end{equation*}
    This implies in particular
    \begin{equation*}
        T^N  \overset{a.s.}{\underset{N \to \infty}{\longrightarrow}} T_{\sqrt{c_3/2}} ~~\mbox{with, for}~a \in (0 , \infty),~ T_a := \inf \left\{ t \geq 0 \, : \, |B_t| \geq a \right\}
    \end{equation*}
    (N.B. we use implicitly here that, for the Brownian motion and for any $a \in (0 , \infty)$, $T_{a + \varepsilon}$ converges almost surely to $T_{a}$ as $\varepsilon \to 0$).
    Combining the two previous displays and recognising a Riemann sum, we deduce the almost sure convergence
    \begin{equation} \label{eq:almostsureconvergencecriticalcase}
        \frac{\ln N}{N^2} \sum_{k \in \frac{\ln N}{N^2} \mathbb{N} } h\left(X^N_k\right) \mathbf{1}_{\{ k \leq T^N\}} \overset{a.s.}{\underset{N \to \infty}{\longrightarrow}} \int_0^{T_{\sqrt{c_3/2}}} h \left(  \sqrt{\frac{2}{c_3}} B_t \right) dt.
    \end{equation}
    Additionally, we note that, by the convergence~\eqref{eq:invarianceprinciplecriticalcase} (and the definition of the random walk), the random variables $(X_{k+1}^N - X_{k}^N)_{k \in \N}$ are independent and have positive probability to be larger than $1$ (at least for $N$ sufficiently large). This implies that the stopping times $(T^N)_{N \in \N}$ have exponential moments (uniform over $N$), and in particular
    \begin{equation} \label{eq:L2upperbound}
        \sup_{N \in \N} \bbE \left[ (T^N)^2 \right] < \infty.
    \end{equation}
    Using that the function $h$ is bounded and denoting by $\left\| h \right\|_\infty$ its $\mathbb{L}^\infty$-norm, we deduce that
    \begin{equation*}
        \left| \frac{\ln N}{N^2} \sum_{k \in \frac{\ln N}{N^2} \mathbb{N} } h\left(X^N_k\right) \mathbf{1}_{\{ k \leq T^N\}}  \right| \leq \left\| h \right\|_\infty T^N.
    \end{equation*}
    Combined with~\eqref{eq:L2upperbound}, we obtain that the collection of random variables on the left-hand side of the previous display is uniformly integrable. Combining this observation with the almost-sure convergence~\eqref{eq:almostsureconvergencecriticalcase}, we obtain the convergence of the expectation
    \begin{equation*}
            c_3^{-1} \mathbb{E} \left[ \frac{\ln N}{N^2} \sum_{k \in \frac{\ln N}{N^2} \mathbb{N} } h\left(X^N_k\right) \mathbf{1}_{\{ k \leq T^N\}}\right] \underset{N \to \infty}{\longrightarrow}  c_3^{-1} \mathbb{E} \left[ \int_0^{T_{\sqrt{c_3/2}}} h\left(  \sqrt{\frac{2}{c_3}} B_t \right) dt \right].
    \end{equation*}
    A time-change of the Brownian motion shows the identity
    \begin{equation*}
     c_3^{-1} \mathbb{E} \left[ \int_0^{T_{\sqrt{c_3/2}}} h\left(  \sqrt{\frac{2}{c_3}} B_t \right) dt \right] = \mathbb{E} \left[ \int_0^{T_{1/\sqrt{2}}} h\left(  \sqrt{2} B_t \right) dt \right] = \int_{-1}^1  h \left( x \right) G (x) \, dx 
    \end{equation*}
    Combining the two previous displays with the identity~\eqref{eq:identititygreensrandomwalk} completes the proof of Proposition~\ref{prop:convergenceoccupationmeasure} in the case $\alpha = 3$.\\
    
    In the case $\alpha > 3$, Theorem 1 of~\cite[Section 23]{spitzer1976principles} states that 
    \begin{equation} \label{eq:unifconvgreensalpha>3}
        \sup_{x \in \Lambda_N} \left| \frac{G_{N , \alpha} \left(0,x\right)}{N} - \frac{G\left(\frac{x}{N}\right)}{\sigma_\alpha} \right| \underset{N \to \infty}{\longrightarrow} 0,
    \end{equation}
    which implies the result (and in fact implies the convergence stated in Proposition~\ref{prop:convergencesquareGreens}).
\end{proof}

\subsection{Decay of the Fourier coefficients of the Green's function} \label{section5.2}

For each $N \in \N$, we introduce the Fourier basis: for $k \in \{-N , \ldots, N\}$
\begin{equation*}
    \forall x \in \{-N , \ldots, N \}, ~~ e_k^N (x) := e^{2i \pi k \frac{x}{2N+1}}.
\end{equation*}
In the following proposition we prove a quantitative estimate on the Fourier coefficients of the discrete Green's function $ G_{N , \alpha}$.

\begin{proposition} \label{prop:prop5.3}
    There exists a constant $C := C(\alpha) < \infty$ such that, for any $N \in \N$ and any $k \in \{-N , \ldots, N\}$,
\begin{equation} \label{eq:boundhighfrequencyGreens}
    \left|  \sum_{x \in \Lambda_N} e_k^N (x)  G_{N , \alpha}(0,x) \right| \leq 
    \left\{ \begin{aligned}
    \frac{CN^{\alpha - 1}}{|k|_+^{\alpha -1}} &~~ \mbox{for}~ \alpha \in (1, 2), \\
    \frac{C N \ln |k|_+}{|k|_+} & ~~\mbox{for}~ \alpha = 2, \\
    \frac{C N^{\alpha - 1}}{|k|_+} & ~~ \mbox{for}~ \alpha \in (2, 3), \\
    \frac{C N^{2}}{(\ln N) |k|_+} & ~~ \mbox{for}~ \alpha = 3, \\
    \frac{C N^{2}}{|k|_+} & ~~ \mbox{for}~ \alpha \in (3, \infty].
    \end{aligned} \right.
\end{equation}
\end{proposition}

\begin{proof}[Proof of Proposition~\ref{prop:prop5.3}]

To prove the inequality~\eqref{eq:boundhighfrequencyGreens}, we distinguish three cases: whether $\alpha > 2$, $\alpha < 2$ or $\alpha = 2$.

\underline{Case 1: $\alpha > 2$.} We first perform a summation by parts and write
\begin{equation} \label{eq:IPPscalinglim}
     \sum_{x \in \Lambda_N} e_k^N (x)  G_{N , \alpha}(0,x) = - \sum_{x \in \Lambda_N} \mathcal{I}_{k}^N(x) \times N\left[G_{N , \alpha}(0,x+1) -  G_{N , \alpha}(0,x)\right].
\end{equation}
where
$
\mathcal{I}_{k}^N(x) := \frac{1}{N} \sum_{y = -N}^x e_k^N(y)
$
is the discrete integral of the Fourier mode $e_k^N$. We next claim that the following upper bound holds: for $k \in \{-N , \ldots, N\}$ and any $x \in \Lambda_N$,
\begin{equation} \label{eq:integralfouriermode}
\left|\mathcal{I}_{k}^N(x) \right| \leq \frac{C}{|k|+1}.
\end{equation}
This inequality is immediate when $k = 0$. For $k \in \{-N , \ldots, N\} \setminus \{0 \}$, it follows from the computation
\begin{equation*}
    \left|\mathcal{I}_{k}^N(x) \right|  = \frac{1}{N} \left| \frac{e^{2i \pi k \frac{-N}{2N+1}} - e^{2i \pi k \frac{x + 1}{2N+1}}}{e^{2i \pi \frac{k}{2N+1}} - 1} \right| \leq \frac{2}{N} \left| \frac{1}{e^{2i \pi \frac{k}{2N+1}} - 1} \right| \leq \frac{C}{|k|}.
\end{equation*}
Combining~\eqref{eq:IPPscalinglim} and~\eqref{eq:integralfouriermode}, we deduce that
\begin{equation*}
    \left| \sum_{x \in \Lambda_N} e_k^N (x)  G_{N , \alpha}(0,x) \right| \leq \frac{C N}{|k|+1} \sum_{x \in \Lambda_N} \left|  G_{N , \alpha}(0,x+1) -  G_{N , \alpha}(0,x)  \right|.
\end{equation*}
The properties of the Green's function stated in Proposition~\ref{prop:propgreensfunction} imply that
\begin{align*}
    N \sum_{x \in \Lambda_N} \left|  G_{N , \alpha}(0,x+1) -  G_{N , \alpha}(0,x)  \right| & = N \sum_{x = -N}^0  (G_{N , \alpha}(0,x+1) -  G_{N , \alpha}(0,x)) \\
    & \quad + N \sum_{x = 0}^N  (G_{N , \alpha}(0,x) -  G_{N , \alpha}(0,x+1)) \\
    & = 2N  G_{N , \alpha}(0,0).
\end{align*}
Combining the two previous inequalities with the upper bound on the value $G_{N , \alpha}(0,0)$ stated in Proposition~\ref{prop:propgreensfunction} completes the proof of the inequality~\eqref{eq:boundhighfrequencyGreens} in the case $\alpha > 2.$ \\

\underline{Case 2: $\alpha < 2$.} In this range of exponent, we decompose the sum by writing
\begin{multline} \label{eq:identitycase2Scalinglim}
    \sum_{x \in \Lambda_N} e_k^N (x)  G_{N , \alpha}(0,x) =  \sum_{x = -N}^{- N/|k| - 1} e_k^N (x)  G_{N , \alpha}(0,x) \\
    +  \sum_{x = - N/|k|}^{N/|k|} e_k^N (x)  G_{N , \alpha}(0,x) +  \sum_{x = N/|k|+1}^{N} e_k^N (x)  G_{N , \alpha}(0,x) .
\end{multline}
We then estimate the three terms on the right-hand side. For the first one, we use a summation by parts and write
\begin{multline*} 
     \sum_{x = -N}^{- N/|k| - 1} e_k^N (x)  G_{N , \alpha}(0,x)  = \mathcal{I}_{k}^N(-N/|k|-1) G_{N , \alpha}(0,-N/|k|) \\
     - \sum_{x = -N}^{- N/|k| - 1} \mathcal{I}_{k}^N(x) \times N\left[G_{N , \alpha}(0,x+1) -  G_{N , \alpha}(0,x)\right].
\end{multline*}
The two terms on the right-hand side are estimated by using the upper bound~\eqref{eq:integralfouriermode} together with the upper bound on the fractional Green's function stated in Proposition~\ref{prop:propgreensfunction}, so
\begin{align*}
    \lefteqn{\left| \sum_{x = -N}^{- N/|k| - 1} e_k^N (x)  G_{N , \alpha}(0,x) \right|} \qquad & \\ &
    \leq \frac{C}{|k|} \frac{1}{(N/|k|)^{2 - \alpha}} + \frac{CN}{|k|} \sum_{x = -N}^{- N/|k| - 1} \left|  G_{N , \alpha}(0,x+1) -  G_{N , \alpha}(0,x)  \right| \\
    & \leq  \frac{C}{N^{2-\alpha}|k|^{1 - \alpha}} + \frac{CN}{|k|}  G_{N , \alpha}(0,-N/|k|) \\
    & \leq \frac{C}{N^{2-\alpha}|k|^{1 - \alpha}} + \frac{CN}{|k|^{1 - \alpha}}  \frac{1}{(N/|k|)^{2 - \alpha}}\\
    & \leq \frac{CN^{\alpha-1}}{|k|^{1 - \alpha}}.
\end{align*}
The third term on the RHS of~\eqref{eq:identitycase2Scalinglim} is estimated using the same computation to get
\begin{equation*}
    \left| \sum_{x = N/|k| + 1}^{N} e_k^N (x)  G_{N , \alpha}(0,x) \right| \leq \frac{CN^{\alpha-1}}{|k|^{1 - \alpha}}.
\end{equation*}
There remains to study the second term on the right-hand side of~\eqref{eq:identitycase2Scalinglim}. We use the upper bound $\left|e_k^N (x)\right| \leq 1$ together with the result of Proposition~\ref{prop:propgreensfunction} to write
\begin{align*}
    \left| \sum_{x = - N/|k|}^{N/|k|} e_k^N (x)  G_{N , \alpha}(0,x) \right| & \leq  \sum_{x = - N/|k|}^{N/|k|} G_{N , \alpha}(0,x) \\
    & \leq \sum_{x = - N/|k|}^{N/|k|}\frac{C}{(1+|x|)^{2 - \alpha}} \\
    & \leq  \frac{C N^{\alpha - 1}}{|k|^{\alpha -1}}.
\end{align*}
Combining the three previous displays with~\eqref{eq:identitycase2Scalinglim} completes the proof of~\eqref{eq:boundhighfrequencyGreens} in the case $\alpha \in (1,2).$\\

\underline{Case 3: $\alpha = 2$.} This case is similar to the previous one; the only difference lies is the asymptotic behaviour of the Green's function. We start with the identities
\begin{multline} \label{eq:identitycase2Scalinglim2}
    \sum_{x \in \Lambda_N} e_k^N (x)  G_{N , \alpha}(0,x) =  \sum_{x = -N}^{- N/|k| - 1} e_k^N (x)  G_{N , \alpha}(0,x) \\
    +  \sum_{x = - N/|k|}^{N/|k|} e_k^N (x)  G_{N , \alpha}(0,x) +  \sum_{x = N/|k|+1}^{N} e_k^N (x)  G_{N , \alpha}(0,x) .
\end{multline}
and
\begin{multline*}
     \sum_{x = -N}^{- N/|k| - 1} e_k^N (x)  G_{N , \alpha}(0,x)  = \mathcal{I}_{k}^N(-N/|k|-1) G_{N , \alpha}(0,-N/|k|) \\
     - \sum_{x = -N}^{- N/|k| - 1} \mathcal{I}_{k}^N(x) \times N \left[G_{N , \alpha}(0,x+1) -  G_{N , \alpha}(0,x) \right].
\end{multline*}
We now use the upper bound~\eqref{eq:integralfouriermode} and the properties of the Green's function stated in Proposition~\ref{prop:propgreensfunction} to obtain
\begin{equation*}
     \left| \sum_{x = -N}^{- N/|k| - 1} e_k^N (x)  G_{N , \alpha}(0,x) \right| \leq \frac{C}{|k|} \ln |k| + \frac{CN}{|k|}  G_{N , \alpha}(0,-N/|k|)  \leq \frac{CN }{|k|} \ln |k|.
\end{equation*}
A similar argument gives
\begin{equation*}
     \left| \sum_{x =  N/|k| + 1}^{N} e_k^N (x)  G_{N , \alpha}(0,x) \right| \leq \frac{C}{|k|} \ln |k| + \frac{CN}{|k|}  G_{N , \alpha}(0,N/|k|)  \leq \frac{CN }{|k|} \ln |k|.
\end{equation*}
Finally, we write
\begin{align*}
    \left| \sum_{x = - N/|k|}^{N/|k|} e_k^N (x)  G_{N , \alpha}(0,x) \right|  \leq  \sum_{x = - N/|k|}^{N/|k|} G_{N , \alpha}(0,x) & \leq C \sum_{x = - N/|k|}^{N/|k|}  \ln \left( N/(1+|x|)\right) \\
    & \leq  \frac{CN }{|k|} \ln |k|.
\end{align*}
Combining the three previous displays with~\eqref{eq:identitycase2Scalinglim2} completes the proof of~\eqref{eq:boundhighfrequencyGreens} in the case $\alpha =2.$
\end{proof}

\subsection{Proof of the convergence of the $\mathbb{L}^2$-norm of the discrete Green's function} \label{section5.3}

\begin{proof}[Proof of Proposition~\ref{prop:convergencesquareGreens}]
To simplify the notation, we only prove the result in the case $\alpha \in (3/2, 3)$. The argument is the same in the cases $\alpha = 3$ and $\alpha \in (3, \infty]$ but the proof differs on a notational level due to the different limits and scaling factors.

By the (discrete) Plancherel identity, we have, for any $N \in \N$,
\begin{equation} \label{eq:plancherel1}
    \sum_{x\in \Lambda_N} G_{N , \alpha}(0,x)^2 = \frac{1}{(2N+1)} \sum_{k = -N}^N \left| \sum_{x \in \Lambda_N} e_k^N (x)  G_{N , \alpha}(0,x) \right|^2.
\end{equation}
and 
\begin{equation} \label{eq:plancherel2}
    \int_{-1}^{1} G_{\alpha}(x)^2 \, dx = \frac{1}{2} \sum_{k \in \mathbb Z} \left| \int_{-1}^{1} e^{i k \pi x} G_\alpha(x) \, dx \right|^2.
\end{equation}
By Proposition~\ref{prop:convergenceoccupationmeasure} (with the function $h(x) = e^{i \pi k x}$), we have, for any $k \in \mathbb{Z}$,
\begin{equation*}
    \frac{1}{N^{\alpha-1}} \sum_{x \in \Lambda_N} e^{i \pi k \frac{x}{N}}  G_{N , \alpha}(0,x)  \underset{N \to \infty}{\longrightarrow}  \int_{-1}^{1} e^{i \pi k x } G_{\alpha}(x) \, dx.
\end{equation*}
Additionally, using that the function $x \mapsto e^{i x}$ is $1$-Lipschitz, we have
\begin{align*}
    \frac{1}{N^{\alpha-1}} \left| \sum_{x \in \Lambda_N} \left( e^{i \pi k \frac{x}{N}} - e^{2 i \pi k \frac{x}{2N+1}} \right)  G_{N , \alpha}(0,x) \right| & \leq \frac{C |k|}{N^{\alpha-1}}   \sum_{x \in \Lambda_N} \left| x \right| \left( \frac{1}{N} - \frac{2}{2N+1} \right)  G_{N , \alpha}(0,x) \\
    & \leq  \frac{C |k|}{N^{\alpha-1}}   \sum_{x \in \Lambda_N} \frac{\left| x \right|}{N^2}  G_{N , \alpha}(0,x) \\
    & \leq  \frac{C |k|}{N}   \underset{\leq C}{\underbrace{\frac{1}{N^{\alpha-1}}  \sum_{x \in \Lambda_N} G_{N , \alpha}(0,x)}} \\
    & \leq  \frac{C |k|}{N},
\end{align*}
where in the second line we used the inequality $|x| \leq N$ for any $x \in \Lambda_N$ and in the third inequality we used the upper bound on the Green's function stated in Proposition~\ref{prop:propgreensfunction} (or Proposition~\ref{prop:convergenceoccupationmeasure} with the function $h = 1$). A combination of the two previous displays shows that, for any integer $k \in \mathbb Z$,
\begin{equation} \label{eq:convergenceFouriermodes}
    \frac{1}{N^{\alpha-1}} \sum_{x \in \Lambda_N} e_k^N(x)  G_{N , \alpha}(0,x)  \underset{N \to \infty}{\longrightarrow}  \int_{-1}^{1} e^{i \pi k x } G_{\alpha}(x) \, dx.
\end{equation}
To ease the notation in the rest of the proof, let us denote by, for any $N \in \N$ and any $k \in \mathbb{Z}$,
\begin{equation*}
    a_k^N :=
    \begin{cases}
    \displaystyle\frac{1}{\sqrt{2N+1}} \frac{1}{N^{\alpha - \frac32}} \sum_{x \in \Lambda_N} e_k^N (x)  G_{N , \alpha}(0,x) & ~~\mbox{if} ~~ k \in \{-N , \ldots, N\} \\
    0 & ~~\mbox{if}  ~~ k \in \mathbb{Z} \setminus \{-N , \ldots, N\}
    \end{cases}
\end{equation*}
and, for any $k \in \N$,
\begin{equation*}
    a_k := \frac{1}{\sqrt{2}} \int_{-1}^{1} e^{i \pi k x } G_{\alpha}(x) \, dx.
\end{equation*}
Using this notation, the Plancherel identity~\eqref{eq:plancherel1} can be rewritten as follows
\begin{equation*}
    \frac{1}{N^{2\alpha -3}}\sum_{x\in \Lambda_N} G_{N , \alpha}(0,x)^2 =  \sum_{k \in \mathbb{Z}} (a_k^N)^2.
\end{equation*}
By~\eqref{eq:convergenceFouriermodes}, we have, for any $k \in \mathbb Z$
\begin{equation*}
    (a_k^N)^2 \underset{N \to \infty}{\longrightarrow} a_k^2,
\end{equation*}
and, by Proposition~\ref{prop:prop5.3}, for any $N \in \N$,
\begin{equation*}
    (a_k^N)^2 \leq
    \begin{cases} 
        \displaystyle \frac{C}{|k|_+^{2 \alpha -2}} &~~\mbox{if} ~~ \alpha \in (3/2, 2), \\
        \displaystyle \frac{C (\ln |k|_+)^2}{|k|_+^{2}} &~~\mbox{if} ~~ \alpha = 2, \\
        \displaystyle \frac{C}{|k|_+^{2}} &~~\mbox{if} ~~ \alpha \in (2 , 3).
    \end{cases}
\end{equation*}
In each case, we have respectively
\begin{equation*}
    \sum_{k \in \mathbb{Z}} \frac{1}{|k|_+^{2 \alpha -2}} < \infty, \hspace{5mm} \sum_{k \in \mathbb{Z}} \frac{(\ln |k|_+)^2}{|k|_+^{2}} < \infty ~~\mbox{and}~~ \sum_{k \in \mathbb{Z}} \frac{1}{|k|_+^{2}} < \infty.
\end{equation*}
(N.B. we use here that if $\alpha > 3/2$, then $2 \alpha - 2 > 1$).
Combining the four previous displays with the Plancherel identity~\eqref{eq:plancherel2} and the dominated convergence theorem, we deduce that
\begin{equation*}
    \frac{1}{N^{2\alpha -3}}\sum_{x\in \Lambda_N} G_{N , \alpha}(0,x)^2 =  \sum_{k \in \mathbb{Z}} (a_k^N)^2 \underset{N \to \infty}{\longrightarrow} \sum_{k \in \mathbb{Z}} (a_k)^2 = \int_{-1}^{1} G_{\alpha}(x)^2 \, dx .\qedhere
\end{equation*}
\end{proof}

\section{Quantitative delocalisation of the height at the origin} \label{sec:quantdelocheightorigin}
The purpose of this section is to prove Theorem \ref{thm:deloc}. The strategy is to show that for all $\alpha>3/2$ and for any realisation of disorder $\eta$, the distance (in $\mathbb{L}^2(\Gib)$-norm) between the height function and the real-valued ground state at 0 is negligible compared to the typical size of the real-valued ground state at 0.
This is the content of Proposition~\ref{eq:distanceheighttogroundstate} below.
Since we identified in Proposition~\ref{prop:convergencesquareGreens} the scaling limit of the ground state at the origin, we easily deduce the first three items of Theorem \ref{thm:deloc}.
The fourth one follows as a by-product of Proposition~\ref{eq:distanceheighttogroundstate}, but the conclusion differs as this time the error term is greater than the typical height of the ground state.

\begin{proposition} \label{prop:distanceheighttogroundstate}
For any decay exponent $\alpha \in (1 , \infty]$ and any inverse temperature $\beta > 0$, there exists a constant $C := C(\alpha, \beta) < \infty$ such that, for any realisation of the disorder $\eta$, any disorder strength $\lambda > 0$ and any $N \in \N$
\begin{equation} \label{eq:distanceheighttogroundstate}
    \left\Vert \phi(0)-\lambda G_{N,\alpha} \eta(0) \right\Vert_{\mathbb{L}^2(\Gib)} \leq C \begin{cases}
        \ln |N|_+ & \mbox{ for } \alpha \in (1,2), \\
        (\ln |N|_+)^2 & \mbox{ for } \alpha =2, \\
        N^{\alpha-2} & \mbox{ for } \alpha \in (2,3), \\
        N/\ln |N|_+ & \mbox{ for } \alpha =3,\\
        N & \mbox{ for } \alpha >3.
    \end{cases}
\end{equation}
\end{proposition}
We stress that in \eqref{eq:distanceheighttogroundstate}, the constant $C$ depends on $\alpha$ and $\beta$ but \emph{not} on $\eta$ (nor $\lambda$). This will be crucial in the proof since we will integrate the RHS of \eqref{eq:distanceheighttogroundstate} over $\eta$.

We first deduce Theorem \ref{thm:deloc} from Proposition \ref{prop:distanceheighttogroundstate}, and delay the proof of the latter to Section \ref{section:Control of the error term/distance to the ground state}.

\begin{proof}[Proof of Theorem \ref{thm:deloc}]
    To prove the first three items of Theorem~\ref{thm:deloc}, we note that by combining Proposition~\ref{prop:distanceheighttogroundstate} and Proposition~\ref{prop:convergencesquareGreens} (N.B. we use that the size of the discrete $\mathbb{L}^2$-norm of the square of the Green's function is always larger than the right-hand side of~\eqref{eq:distanceheighttogroundstate} and that this right-hand side does not depend on the disorder $\eta$), we have, for any $\alpha \in (3/2, \infty)$,
    \begin{equation} \label{eq:convergenceto0}
        \frac{ \mathbb{E} \E \left[ \left( \phi(0)-\lambda G_{N,\alpha} \eta(0) \right)^2 \right]}{\sum_{x\in \Lambda_N} G_{N , \alpha}(0,x)^2}  \underset{N \to \infty}{\longrightarrow} 0.
    \end{equation}
    Combined with~\eqref{eq:sumandvarianceiid}, the previous convergence implies
    \begin{equation*}
        \bbE \bigl[\E[\phi(0)^2]\bigl]= \lambda^2\bbE\left[G_{N, \alpha}\eta(0)^2\right] +o_{N \to \infty}\left(\bbE\left[G_{N, \alpha}\eta(0)^2\right]\right)
    \end{equation*}
    and thus
    \begin{equation*}
        \frac{\mathbb{E} \E \left[ \phi(0)^2 \right]}{\sum_{x\in \Lambda_N} G_{N , \alpha}(0,x)^2} \underset{N \to \infty}{\longrightarrow} \lambda^2.
    \end{equation*}
    The asymptotics on the $\mathbb{L}^2$-norm of the Green's function established in Proposition~\ref{prop:convergencesquareGreens} then yield the the first three items of Theorem~\ref{thm:deloc}.

    The fourth item is a direct consequence of Proposition \ref{prop:distanceheighttogroundstate}, which guarantees that for all $\alpha\in(1,2)$,
    \begin{equation*}
        \Vert \phi(0) \Vert_{\mathbb{L}^2(\Gib)} \leq C \ln |N|_+ + \lambda |G_{N, \alpha}\eta(0)|.
    \end{equation*}
    The result in the case $\alpha \in (1,3/2)$ follows, since we have
    $$\bbE\left[G_{N, \alpha}\eta(0)^2\right] \leq C \begin{cases}
        1 & \mbox{ if } \alpha<3/2 \\
        \ln |N|_+ & \mbox{ if } \alpha=3/2
    \end{cases}$$
    (this is a simple computation using \eqref{eq:sumandvarianceiid} and the bounds given in Proposition \ref{prop:propgreensfunction}).
\end{proof}

\subsection{Control of the distance to the ground state}\label{section:Control of the error term/distance to the ground state}
This section is devoted to the proof of Proposition \ref{prop:distanceheighttogroundstate}.
Throughout the section, we fix $\alpha \in(1, \infty]$, $\beta>0$, $\lambda>0$, a realisation $\eta$ of the disorder and $N \in \N$. However, we will make sure that none of the constants depend on $\lambda, \eta$ or $N$.

We first notice that since $G_{N, \alpha} (-L_{N, \alpha})=I_{\Lambda_N}$ (see \eqref{eq:def2GSrealvaluedsec2}),
$$\lambda G_{N, \alpha} \eta(0)-\phi(0) = \sum_{x \in \Lambda_N} G_{N, \alpha}(0,x) [\lambda \eta_x+L_{N, \alpha}\phi(x)]=\sum_{x \in \Lambda_N} G_{N, \alpha}(0,x) \rho_N^{\eta, \phi}(x),$$
where $\rho_{N}^{\eta,\phi}(x) := \lambda\eta_x +L_{N, \alpha}\phi(x)$. 
The idea to estimate the $\mathbb{L}^2$-norm of this quantity will be to use summation by parts to get a telescopic sum.
Thus, we will need to control the sum of $\rho^{\eta, \phi}_N(x)$ for $x$ lying in some intervals.
This is the purpose of the following lemma.

\begin{lemma} For every $x \in \Lambda_{N}, \phi \in \Omega_{N}$, let $\rho_{N}^{\eta,\phi}(x) := \lambda\eta_x +L_{N, \alpha}\phi(x)$.
Then, for every $B \subseteq \Lambda_{N}$, we have
\begin{equation}\label{eq:error_first_moment}
\E \left| \sum_{x \in B}  \rho_N^{\eta,\phi}(x) \right| \leq \frac{1}{2} \Cap(B)+\frac{1}{\beta} \ln(2)
\end{equation} and
\begin{equation}\label{eq:error_second_moment}
\E\left( \sum_{x \in B} \rho_N^{\eta,\phi}(x) \right)^2 \leq \frac12 \Cap(B)^2 + C(\beta) \Cap(B)
\end{equation}
for some constant $C(\beta) \in (0, \infty)$ that only depends on $\beta$.
\end{lemma}

\begin{proof}
Let $\phi$ be distributed according to $\Gib$. For some $\mathbb Z-$valued random variable $K$, measurable with respect to $\phi$ (to be determined later), consider $\widetilde{\phi} = \phi + K \mathbf{1}_B$; call $\nu$ the law of $\widetilde \phi$. The strategy of the proof is to apply the Boltzmann--Gibbs variational principle \eqref{eq:BGVP} to $\nu$ and $\Gib$, and then to suitably choose $K$ to conclude.

We start by computing the difference of energy:
\begin{align*}
\mathcal H_{N}^\eta(\widetilde\phi) - \mathcal H_{N}^\eta(\phi) &= -\frac{1}{2}\langle \phi + K \mathbf{1}_B, L_{N, \alpha}(\phi + K \mathbf{1}_B) \rangle + \frac{1}{2}\langle \phi, L_{N, \alpha}\phi \rangle -\langle \phi + K \mathbf{1}_B,\lambda \eta \rangle + \langle \phi, \lambda\eta \rangle \\
&= -K \langle\mathbf{1}_B,L_{N, \alpha}\phi + \lambda \eta \rangle - \frac{1}{2} K^{2} \langle \mathbf{1}_B, L_{N, \alpha} \mathbf{1}_B \rangle \\
&= -K \sum_{x \in B} \rho_N^{\eta,\phi}(x) + \frac{1}{2} K^{2} \Cap(B).
\end{align*}
Therefore, by \eqref{eq:BGVP}, we obtain
\begin{align*}
\Ent(\nu) - \Ent(\Gib) &\leq \beta (\nu[H_N^{\eta}] - \Gib[H_N^{\eta}]) \\
&=\beta \E[H_N^{\eta}(\widetilde{\phi}) - H_N^{\eta}(\phi)] \\
&=- \beta \E \left[ K \sum_{x \in B} \rho_N^{\eta,\phi}(x) \right] + \frac{\beta}{2} \Cap(B) \E[K^{2}].
\end{align*}
Hence,
\begin{equation}\label{eq:result_from_BGVP}
    \E \left[ K \sum_{x \in B} \rho_N^{\eta,\phi}(x) \right] \leq \frac{1}{2}\Cap(B)\E[K^{2}] + \frac{1}{\beta}(\Ent(\Gib) - \Ent(\nu)).
\end{equation}
Since $\widetilde \phi$ is a measurable function of $\phi$, we have
$$\Ent(\Gib)-\Ent(\nu) = \Ent( \phi \mid \widetilde \phi) = \Ent( K \mid \widetilde \phi),$$
the latter coming from the fact given $\widetilde \phi$, $\phi$ is determined by $K$ and vice versa (one has $\Ent(\phi \mid \widetilde \phi, K)=\Ent(K \mid \widetilde \phi, \phi)=0$).

\medskip
\underline{Proof of \eqref{eq:error_first_moment}:} We choose $K = \begin{cases} 1 & \text{if } \sum_{x \in B}\rho_N^{\eta,\phi}(x) \ge 0 \\ -1 & \text{otherwise} \end{cases},$

so that \eqref{eq:result_from_BGVP} becomes
\[
\E \left| \sum_{x \in B} \rho_N^{\eta,\phi}(x) \right| \le \frac{1}{2} \Cap(B) + \frac{\Ent(K\mid \widetilde \phi)}{\beta}.
\]

Then, since $K$ is supported only on a set of cardinal 2, we have $\Ent(K \mid \widetilde \phi) \leq \ln 2$.
Indeed, for any $y \in \Omega_N$, $\Ent(K \mid \widetilde{\phi}=y)$ is of the form $-p \ln p - (1-p) \ln(1-p)$, where $p$ is the probability of $K=1$ given $\widetilde \phi=y$, and this is maximal when $p=1/2$.
This proves \eqref{eq:error_first_moment}.

\medskip 
\underline{Proof of \eqref{eq:error_second_moment}:} We define an odd version of the integer part by $\lceil t \rfloor:= \begin{cases} \lfloor t \rfloor & \text{if }t\ge 0 \\ \lceil t \rceil & \text{if }t<0 \end{cases}$, for all $t\in \mathbb{R}$ (clearly, $\lceil -t \rfloor=-\lceil t \rfloor$).
We shall use the following two inequalities verified by our integer part function:
\begin{enumerate}
\item[(a)] $| \lceil t \rfloor| \leq |t|$ and $ \lceil t \rfloor^2 \leq t  \lceil t \rfloor$ \quad (noticing that $\lceil t \rfloor$ and $t$ have the same sign);
\item[(b)] $t^2 - |t| \leq \lceil t \rfloor t$ \quad (it is clear for $t\geq0$, so it follows by parity).
\end{enumerate}
We choose
\[
K =\left\lceil \delta \sum_{x \in B} \rho_N^{\eta,\phi}(x) \right\rfloor \quad \quad \text{ with }\delta := (\Cap B)^{-1} \in (0,1).
\]
Then, using (a) to upper bound $\E[K^2]$, equation \eqref{eq:result_from_BGVP} gives
\begin{align*}
\E \left[ K \sum_{x \in B} \rho_N^{\eta,\phi}(x) \right] &\leq \frac{1}{2} \Cap(B) \E \left[ K \cdot \delta \sum_{x\in B} \rho_N^{\eta,\phi}(x) \right] + \frac{\Ent(K \mid \widetilde \phi)}{\beta} \\
& \leq \frac{1}{2} \E \left[ K \sum_{x\in B} \rho_N^{\eta,\phi}(x) \right] + \frac{\Ent(K \mid \widetilde \phi)}{\beta}, 
\end{align*}
This can be rewritten as
$$\E \left[ K \sum_{x \in B} \rho_N^{\eta,\phi}(x) \right] \leq  \frac{2 \Ent(K \mid \widetilde \phi)}{\beta}.$$
We then apply successively (b) and \eqref{eq:error_first_moment} to get a lower bound for the LHS.
\begin{align*}
\E \left[ K \sum_{x \in B} \rho_N^{\eta,\phi}(x) \right] &\geq \frac{1}{\delta} \E \left[ \delta^2 \left(\sum_{x \in B} \rho_N^{\eta,\phi}(x)\right)^2 \right] - \frac{1}{\delta} \E \left| \delta \sum_{x \in B} \rho_N^{\eta,\phi}(x)\right| \\
&\geq \delta \E\left(\sum_{x \in B} \rho_N^{\eta,\phi}(x)\right)^2  - \dfrac{\ln(2)}{\beta} - \dfrac{\Cap(B)}{2}
\end{align*}
Putting everything together and multiplying by $\Cap(B)=\delta^{-1}$, we obtain
\begin{equation}\label{eq:stepto_error_second_moment}
\E\left(\sum_{x \in B} \rho_N^{\eta,\phi}(x)\right)^2 \leq \frac12 \Cap(B)^2 + \frac{\ln 2 }{\beta} \Cap(B) + \frac{2}{\beta} \Cap(B)\Ent(K \mid \widetilde \phi).
\end{equation}
Once again, it only remains to control the entropy term.

Fix $\widetilde y$ and consider the conditional law of $K$ given $\widetilde\phi = \widetilde y$, with conditional mean $m(\widetilde y) = \E[|K| \mid \widetilde\phi = \widetilde y]$. First, since the sign of $K$ takes at most two values,
$$\Ent(K \mid \widetilde\phi = \widetilde y) \leq \Ent(|K| \mid \widetilde\phi = \widetilde y) + \ln 2.$$
To bound $\Ent(|K| \mid \widetilde\phi = \widetilde y)$, let $q$ be the (shifted) geometric law on $\mathbb{N}= \{0,1,2,\ldots\}$ with parameter $\theta = 1/(m(\widetilde y)+1)$, i.e.\ $q(n) = \theta (1-\theta)^n$, which has mean $m(\widetilde y)$ and entropy $\ln(m(\widetilde y)+1) + m(\widetilde y)\ln(1+1/m(\widetilde y))$. It is known that $q$ maximises entropy among all laws on $\mathbb{N}$ with mean $m(\widetilde y)$ (this can be seen as an application of the Boltzmann--Gibbs variational principle on $\Omega= \N$), so
$$\Ent(|K| \mid \widetilde\phi = \widetilde y) \leq \ln(m(\widetilde y)+1) + m(\widetilde y)\ln\left(1+\tfrac{1}{m(\widetilde y)}\right) \leq \ln(m(\widetilde y)+1) + 1.$$
Thus, we get $$\Ent(K \mid \widetilde\phi = \widetilde y) \leq \ln(m(\widetilde y)+1) + 1 + \ln2.$$
 Taking the expectation over $\widetilde\phi$ and applying Jensen's inequality to the concave function $t \mapsto \ln(t+1)$ yields
\begin{align*}
\Ent(\Gib) - \Ent(\nu) = \Ent(K \mid \widetilde\phi)
&\leq \E\bigl[\ln(m(\widetilde\phi)+1)\bigr] + 1 + \ln 2 \\
&\leq \ln\bigl(\E[m(\widetilde\phi)]+1\bigr) + 1 + \ln 2 \\
&= \ln\bigl(\E|K|+1\bigr) + 1 + \ln 2,
\end{align*}
where we used the law of total expectation $\E[m(\widetilde\phi)] = \E|K|$.

Finally, we successively apply (a) and \eqref{eq:error_first_moment} once more to bound $\E|K|$:
$$\E|K| \leq \delta \E \left| \sum_{x \in B} \rho_N^{\eta, \phi}(x) \right| \leq\frac12 + \frac{\delta \ln 2}{\beta}.$$
Plugging in \eqref{eq:stepto_error_second_moment} the bound on the entropy that we obtained concludes the proof:
$$\E\left(\sum_{x \in B} \rho_N^{\eta,\phi}(x)\right)^2 \leq \frac12 \Cap(B)^2+ \frac{\ln(2)}{\beta} \Cap(B) + \frac{2}{\beta} \Cap(B) \left(1+ \ln 2+\ln\left(\tfrac32+ \tfrac{\delta \ln 2}{\beta}\right) \right) $$
(one can take for example $C(\beta)= \frac{1+\ln 6}{\beta}+\frac{4 \ln(2)}{3 \beta^2}$, using the facts that $\ln(3/2+t) \leq \ln(3)-\ln(2)+2t/3$ and $\delta=(\Cap B)^{-1} \leq 1$).
\end{proof}

We may now prove Proposition \ref{prop:distanceheighttogroundstate}.
\begin{proof}[Proof of Proposition \ref{prop:distanceheighttogroundstate}]
    We have already noticed that $$\lambda G_{N, \alpha} \eta(0)-\phi(0) =\sum_{x \in \Lambda_N} G_{N, \alpha}(0,x) \rho_N^{\eta, \phi}(x).$$
    Thus, we have 
    \begin{multline*}
        \left\Vert\phi(0)-\lambda G_{N,\alpha} \eta(0) \right\Vert_{\mathbb{L}^2(\Gib)} \leq \left\Vert\sum_{x=1}^N G_{N, \alpha}(0,x) \rho_N^{\eta, \phi}(x)\right\Vert_{\mathbb{L}^2(\Gib)} \\ \qquad +\left\Vert\sum_{x=-N}^{-1} G_{N, \alpha}(0,x) \rho_N^{\eta, \phi}(x)\right\Vert_{\mathbb{L}^2(\Gib)}+\left\Vert G_{N, \alpha}(0,0) \rho_N^{\eta, \phi}(0)\right\Vert_{\mathbb{L}^2(\Gib)}
    \end{multline*}
    The first two terms are treated the same way, and the last one is, by \eqref{eq:error_second_moment} with $B=\{0\}$, bounded by a constant time $G_{N, \alpha}(0,0)$. Let us consider the sum over $\{1, \dots, N\}$.
    Summing by parts, we get
    \begin{align*}
        \lefteqn{\left(\sum_{x=1}^N G_{N,\alpha}(0,x)\rho_N^{\eta, \phi}(x)\right)^2} \qquad &\\ &= \left(\sum_{x=1}^N [G_{N, \alpha}(0,x)-G_{N, \alpha}(0,x+1)] \sum_{y=1}^x \rho_N^{\eta, \phi}(y)\right)^2 \\
        &\leq \sum_{x=1}^N |G_{N, \alpha}(0,x)-G_{N, \alpha}(0,x+1)|\sum_{x=1}^N |G_{N, \alpha}(0,x)-G_{N, \alpha}(0,x+1)| \left(\sum_{y=1}^x \rho_N^{\eta, \phi}(y)\right)^2,
    \end{align*}
by Cauchy--Schwarz inequality.
By monotonicity of the Green's function, the first sum is equal to $G_{N, \alpha}(0,1)$, which is less than $G_{N, \alpha}(0,0)$. According to \eqref{eq:error_second_moment}, the term $\left(\sum_{y=1}^x \rho_N^{\eta, \phi}(y)\right)^2$ is, in expectation (with respect to $\Gib$), bounded by a constant time $\Cap(\{1, \dots, x\})$, which is at most of order $\Cap(\Lambda_N)$.
By taking expectation in the previous equation, and bounding again the other telescopic sum by $G_{N,\alpha}(0,0)$, we obtain
\begin{equation*}
    \E \left[ \left(\sum_{x=1}^N G_{N,\alpha}(0,x)\rho_N^{\eta, \phi}(x)\right)^2 \right] \leq C G_{N, \alpha}(0,0)^2 \Cap(\Lambda_N)^2.
\end{equation*}
The contribution of $\{-N, \cdots, -1\}$ is of the same order, and the one of 0 is smaller, so
\begin{equation*}
    \left\Vert\phi(0)-\lambda G_{N,\alpha} \eta \right\Vert_{\mathbb{L}^2(\Gib)} \leq C G_{N, \alpha}(0,0) \Cap(\Lambda_N).
\end{equation*}

This is enough to conclude the proof in the case $\alpha \geq 2$. Indeed, for $\alpha>2$, the capacity is bounded by a constant (see \eqref{eq:bound_capacity}) and the upper bounds on $G_{N, \alpha}$ provided by Proposition \ref{prop:propgreensfunction} give the right order; for $\alpha=2$, both $\Cap(\Lambda_N)$ and $G_{N,2}(0,0)$ are of order $\ln N$, hence the result.\footnote{In fact, bounding the error term by $C G_{N,\alpha}(0,0)\Cap(\Lambda_N)$ would be enough to prove Theorem \ref{thm:deloc} for all $\alpha>5/3$, since this quantity is negligible with respect to $N^{2 \alpha -3}$ in this case.}

To get an upper bound of order $\ln N$ in the case $\alpha<2$, we shall divide the interval $\{1, \dots ,N\}$ into dyadic blocks.

For $k \geq 0$, set
$$B_k:= \{2^k, 2^k+1, \cdots , 2^{k+1}-1\} \cap \Lambda_N.$$
These blocks are empty for $k> \lfloor \log_2 N \rfloor$. Thus, we have
\begin{align*}
    \left\Vert\sum_{x=1}^N G_{N, \alpha}(0,x) \rho_N^{\eta, \phi}(x)\right\Vert_{\mathbb{L}^2(\Gib)}  &\leq \sum_{k=0}^{\lfloor \log_2 N\rfloor} \left\Vert\sum_{x\in B_k} G_{N, \alpha}(0,x) \rho_N^{\eta, \phi}(x)\right\Vert_{\mathbb{L}^2(\Gib)} \\
    &\leq \sum_{k=0}^{\lfloor \log_2 N\rfloor} \left\Vert \sum_{x\in B_k} [G_{N, \alpha}(0,x)-G_{N, \alpha}(0,x+1)] \sum_{y=2^k}^x \rho_N^{\eta, \phi}(y)\right\Vert_{\mathbb{L}^2(\Gib)} \\
    &\qquad \qquad + \sum_{k=0}^{\lfloor \log_2(N) \rfloor} G_{N,\alpha}(0,2^{k+1}) \left\Vert \sum_{y=2^k}^{2^{k+1}-1} \rho_N^{\eta, \phi}(y)\right\Vert_{\mathbb{L}^2(\Gib)}
\end{align*}
by performing on each block the same summation by parts as before.
Then, with the same computation as above using Cauchy--Schwarz inequality, we can bound the first term as follows:
    \begin{align*}
        \lefteqn{\E \left(\sum_{x\in B_k} [G_{N, \alpha}(0,x)-G_{N, \alpha}(0,x+1)] \sum_{y=2^k}^x \rho_N^{\eta, \phi}(y)\right)^2} \qquad &\\
        &\leq \E\left[ \sum_{x\in B_k} |G_{N, \alpha}(0,x)-G_{N, \alpha}(0,x+1)|\sum_{x\in B_k} |G_{N, \alpha}(0,x)-G_{N, \alpha}(0,x+1)| \left(\sum_{y=2^k}^x \rho_N^{\eta, \phi}(y)\right)^2\right] \\
        &\leq C G_{N, \alpha}(0,2^k)^2 \Cap(B_k)^2.
    \end{align*}
Combined with the previous equation, and also by applying \eqref{eq:error_second_moment} to the boundary term, this gives
\begin{align*}
    \left\Vert\sum_{x=1}^N G_{N, \alpha}(0,x) \rho_N^{\eta, \phi}(x)\right\Vert_{\mathbb{L}^2(\Gib)}  &\leq C \sum_{k=0}^{\lfloor \log_2 N\rfloor} \Big[ G_{N, \alpha}(0,2^k) \Cap(B_k)+G_{N,\alpha}(0,2^{k+1}) \Cap(B_k) )\Big] \\
\end{align*}
According to the upper bound on the Green's function (Proposition \ref{prop:propgreensfunction}) in the regime $\alpha\in(1,2)$ and the one on the capacity of an interval \eqref{eq:bound_capacity}, we have that all the terms in the above sum are bounded by a constant.
Since there are $O(\ln N)$ terms, this yields the result.
\end{proof}

\section{Central limit theorem} \label{sec:secCLT}
In this short section, we prove Theorem \ref{thm:thm1.4TCL}. Similarly as above, our strategy is to first show a central limit theorem for the real-valued ground state, and then to take advantage of our control on the $\mathbb{L}^2$-distance between the height function and this ground state.

\begin{proof}[Proof of Theorem \ref{thm:thm1.4TCL}]
One has the following convergence in distribution:
\begin{equation} \label{eq:CLTgroundstate}
    \frac{G_{N,\alpha}\eta(0)}{\sqrt{\sum_{x\in \Lambda_N} G_{N , \alpha}(0,x)^2}} \overset{\mathcal{L}}{\underset{N \to \infty}{\longrightarrow}} \mathcal N(0,1).
\end{equation}
This is a direct consequence of the Lindeberg--Feller theorem (see e.g. \cite[Theorem 3.4.10]{Durrett_2019} or \cite[Theorem 27.2]{Bilingsley_1995}) which generalises the central limit theorem to triangular arrays of random variables.
Indeed, for all $N \in \N$, set for all $k\in \{-N, \dots, N\}$, $X_{N,k}:=\sigma_{N,k} \eta_k$ where $\sigma_{N,k}:=\frac{G_{N,\alpha}(0,k)}{\sqrt{\sum_{x\in \Lambda_N} G_{N,\alpha}(0,x)^2}}$ ($\sigma^2_{N,k}$ is thus the variance of $X_{N,k}$). The random variables $(X_{N,k})_{N \in \N, -N \leq k \leq N}$ are therefore independent and we have for all $N$, $\sum_{k=-N}^N \sigma^2_{N,k}=1$. Then, the Lindeberg condition
$$\forall \varepsilon>0, \quad\lim_{N \to \infty} \sum_{k=-N}^N \bbE [X_{N,k}^2 \mathbf{1}_{\{|X_{N,k} |> \varepsilon\}}] \xrightarrow[N\to \infty]{}0$$
holds because the sum over $\{-N, \dots N\}$ is upper bounded by
$\bbE[\eta_0^2 \mathbf{1}_{\{|\eta_0|> \varepsilon/\max_k \sigma_{N,k}\}}]$,
which goes to $0$ as $N\to \infty$ since $\max_k \sigma_{N,k}$ goes to 0 (this follows, for any decay exponent $\alpha>3/2$, from the upper bounds on the Green's function stated in Proposition \ref{prop:propgreensfunction} and the estimates on its $\mathbb L^2$-norm proved in Proposition \ref{prop:convergencesquareGreens}).
Lindeberg--Feller theorem then states that the sum $\sum_{k=-N}^N X_{N,k}$ converges in distribution towards a standard normal random variable, which is exactly \eqref{eq:CLTgroundstate}.

We now combine~\eqref{eq:CLTgroundstate} with~\eqref{eq:convergenceto0} together with the following statement: if $(X_N)_{N \in \N}$ and $(Y_N)_{N \in \N}$ are two sequences of random variables, then
\begin{equation*}
    Y_N \underset{N \to \infty}{\overset{\mathcal{L}}{\longrightarrow}} \mathcal N(0,\lambda^2) ~~\mbox{and}~~X_N - Y_N \underset{N \to \infty}{\overset{\mathbb{L}^2}{\longrightarrow}} 0 ~~ \implies ~~X_N \underset{N \to \infty}{\overset{\mathcal{L}}{\longrightarrow}} \mathcal N(0,\lambda^2).
\end{equation*}
We obtain $$\frac{\phi_N(0)}{\sqrt{\sum_{x \in \Lambda_N}G_{N,\alpha}(0,x)^2}} \overset{\mathcal{L}}{\underset{N \to \infty}{\longrightarrow}} \mathcal N(0,\lambda^2).$$
Then, since 
$$\dfrac{\bbE\E[\phi(0)^2]}{\sum_{x \in \Lambda_N}G_{N,\alpha}(0,x)^2} \to \lambda^2 ~~\mbox{as}~~N\to \infty,$$ it yields the result.
\end{proof}

\appendix
\section{Properties of the killed Green's function}\label{appendix}
The purpose of this appendix is to prove the properties of the killed Green's function listed in Proposition \ref{prop:propgreensfunction}. All of this should be considered as being part of the folklore, and much more precise results are known,\footnote{For example, we have seen that the result of Proposition \ref{prop:convergenceoccupationmeasure} is already contained in \cite{spitzer1976principles} for $\alpha>3$ and in \cite{LSSW17} for $\alpha\in(1,3)$.} but we are unaware of a reference containing especially all the upper bounds for the different regimes of decay exponent.

\subsection{Basic properties}
We start by proving items (1) and (2).
The fact that $G_{N, \alpha}(0, x)=0$ for all $x \notin \Lambda_N$ is by definition ---the random walk being killed outside $\Lambda_N$--- and the fact that $G_{N, \alpha}(0, \cdot)$ is even comes by symmetry of the random walk. 

We may now prove that $G_{N, \alpha}(0, \cdot)$ is non-increasing on $\N$. 
Recall that for all $x\in \Lambda_N$,
$$G_{N, \alpha}(0,x) = \int_0^{+\infty} p_t^{\Lambda_N}(0,x) \mathrm{d}t \qquad \text{with }p_t^{\Lambda_N}(0,x)=\langle \delta_0,e^{t L_{N, \alpha}} \delta_x \rangle.$$
In fact, we will prove that for all $t\geq0$, $p_t^{\Lambda_N}(0, \cdot)$ is non-increasing on $\N$, which immediately implies our claim.
Recall that $c_\alpha=2\sum_{n \geq 1} n^{-\alpha}$ and let us write
$$p_t^{\Lambda_N}=e^{t L_{N, \alpha}}=e^{-(c_\alpha+1) t} e^{tQ} = e^{-(c_\alpha+1)t} \sum_{n \geq0} \frac{t^n}{n!}Q^n\qquad \text{ with } Q=(c_\alpha+1) I+L_{N, \alpha}.$$
Here, one should think of $Q$ as playing the role of a (non-normalised) transition matrix of a lazy version of the discrete-time random walk.
We have for all $x \neq y \in \Lambda_N$,
$$Q(x,x)= 1 \qquad \mbox{ and } \qquad Q(x,y)=|x-y|^{-\alpha}$$

For all $n\geq 0$, we let $u_n(x):=Q^n(0,x)$ for all $x \in \Lambda_N$ and $u_n(\cdot):=0$ outside $\Lambda_N$. We show by induction that $u_n$ is even and non-increasing on $\N$ for all $n \geq0$. It is trivial for $n=0$, and for $n=1$, we have $u_1(0)=1$ and $u_1(x)=|x|^{-\alpha}$ for all $x \in \Lambda_N \setminus \{0 \}$, so the property holds. Assume that it holds for some $n \geq 1$ and set $r(0) := 1$ and $r(x) := |x|^{-\alpha}$ for $x \in \bbZ \setminus \{0\}$; one has for all $x \in \Lambda_N$
$$u_{n+1}(x)= \sum_{y \in \Lambda_N} u_n(y) Q(y,x) = \sum_{y \in \mathbb Z} u_n(y) r(x-y)$$
since $u_n$ vanishes outside $\Lambda_N$.
By parity of $u_n$ and $r$, it is clear that $u_{n+1}$ is still even.
For all $k$, let $\Delta u_n(k)=u_n(k)-u_n(k+1)$.
Fix $x\in \{0, \dots, N-1\}$ and observe that
\begin{align*}
    u_{n+1}(x)-u_{n+1}(x+1)&= \sum_{y \in \mathbb Z} \Delta u_n(y) r(x-y)\\
    &=\sum_{y \geq0} \Delta u_n(y) \big[r(x-y)-r(x+y+1)\big]
\end{align*}
where the last equality relies on the fact that $-\Delta u_n(k)=\Delta u_n(-k-1)$ for all $k$, which holds by parity of $u_n$.
Since $u_n$ and $r$ are non-increasing on $\N$, one has $\Delta u_n(y) \geq0$ for all $y \geq 0$ and $r(x-y) \geq r(x+y+1)$, as $|x-y| \leq x+y+1$ (here we also use the parity of $r$).
This proves that $u_{n+1}(x) \geq u_{n+1}(x+1)$ for $x \in \{0, \dots, N-1\}$, and since $u_{n+1}(x)=0$ for all $x>N$, we deduce that $u_{n+1}$ is non-increasing on $\N$.\footnote{In fact, here we prove the following more general fact: for any two even functions (one of them with compact support) which are non-increasing on a half-line, the convolution of the two is also even and non-increasing on the same half-line; this holds both for continuous or discrete functions.}

Finally, since we have
$$p_t^{\Lambda_N}(0,x) = e^{-(c_\alpha+1) t} \sum_{n \geq 0} \frac{t^n}{n!}u_n(x),$$
we deduce that $p_t^{\Lambda_N}(0, \cdot)$ is non-increasing on $\N$ and that the same property holds for $G_{N, \alpha}(0, \cdot)$.

\subsection{Upper bounds}
In this section, we prove item (3) of Proposition \ref{prop:propgreensfunction}. 
We shall use the following upper bounds on $p_t^{\Lambda_N}(0,x)$, which in fact hold for the non-killed random walk:
\begin{subequations}
\begin{equation}\label{eq:bound_pt(0,x)_alpha<3}
     p_t^{\Lambda_N}(0,x) \leq C \begin{cases}
        (1+t)/|x|_+^\alpha & \text{ if }t< |x|_+^{\alpha-1} \\
        t^{-1/(\alpha-1)}& \text{ if }t\geq |x|_+^{\alpha-1}\end{cases} \qquad \text{ for } \alpha\in (1,3),
\end{equation}
\begin{equation}\label{eq:bound_pt(0,x)_alpha=3}
    p_t^{\Lambda_N}(0,x) \leq C \sqrt{1+t\ln|t|_+}^{-1} \qquad \text{ for } \alpha=3.
\end{equation}
\end{subequations}
The first point is a consequence of Theorem 1.1 in \cite{Bass_Levin} and the second one follows from Theorem 1.1 in \cite{Murugan_Saloff-Coste} (see also \cite[Proposition B.1]{Bou-Rabee_Dario}).

Another crucial tool of the proof will be the fact that the smallest eigenvalue of $-L_{N,\alpha}$, denoted $\lambda^{(1)}_{N, \alpha}$, satisfies
\begin{equation}\label{eq:smallest_eig_val}
    \lambda^{(1)}_{N, \alpha} \geq c \begin{cases}
        N^{1-\alpha} & \text{ for } \alpha \in (1,3) \\
        \ln(N)/N^2 & \text{ for } \alpha=3
    \end{cases}.
\end{equation}
The case $\alpha=3$ is covered by \cite[Proposition 4.1]{Bou-Rabee_Dario}, and it is easily observed for $\alpha \in (1,3)$:
for all $f\in \Omega_N^\R$, one has
\begin{align*}
    \langle f, (-L_{N, \alpha}) f \rangle &=  \sum_{x \neq y} \dfrac{|f(x)-f(y)|^2}{|x-y|^\alpha} \\
    &\geq \sum_{x \in \Lambda_N} \sum_{y=N+1}^{2N}\dfrac{|f(x)-f(y)|^2}{|x-y|^\alpha} \\
    &\geq \frac{1}{(3N)^\alpha}\sum_{x \in \Lambda_N} N f(x)^2 =c N^{1-\alpha} \langle f,f \rangle,
\end{align*}
where, in the last inequality, we used that $f(y)=0$ for $y>N$ and $|x-y| \leq 3N$.

Recall that our goal is to prove
    \begin{equation*}
        G_{N, \alpha}(0,x) \leq C \begin{cases}
            |x|_+^{\alpha-2} & \mbox{ for }\alpha \in (1,2) \\
            1+\ln(N/|x|_+) & \mbox{ for }\alpha =2 \\
            N^{\alpha-2} & \mbox{ for }\alpha \in (2,3) \\
            N/\ln N & \mbox{ for }\alpha =3 \\
            N & \mbox{ for }\alpha >3 \\
        \end{cases}.
    \end{equation*}

\begin{proof}[Proof of Proposition \ref{prop:propgreensfunction}, item (3)]
    We split the proof into 5 cases, depending on the values of $\alpha$, starting by the easier ones. Observe that for $\alpha >2$ (cases 1, 3 and 4 below), we only need to bound $G_{N, \alpha}(0,0)$, whereas for $\alpha \in (1,2]$ (cases 2 and 5), we aim to control the decay of $G_{N, \alpha}(0,x)$ as $|x|$ grows.
    
    \smallskip
    \underline{Case 1 : $\alpha>3$.} This is the content of Proposition 5 in \cite[Section 22]{spitzer1976principles} (in fact, we even noticed above that by Theorem 1 of \cite[Section 23]{spitzer1976principles}, one has $G_{N,\alpha}(0,0)/N \to 1/(2 \sigma_\alpha^2)$ as $N\to \infty$, which is stronger).
    
    \smallskip
    \underline{Case 2 : $\alpha\in(1,2)$.} This case is easy once \eqref{eq:bound_pt(0,x)_alpha<3} is established, essentially because it is the range where the (non-killed) random walk is already transient. For all $x \in \Lambda_N$, it gives
    \begin{align*}
        G_{N, \alpha}(0,x) \leq C \int_0^{|x|_+^{\alpha-1}} \dfrac{1+t}{|x|_+^\alpha} \mathrm{d}t+C \int_{|x|_+^{\alpha-1}}^{+\infty} t^{-1/(\alpha-1)} \mathrm{d}t \leq C |x|_+^{\alpha-2}.
    \end{align*}
    
    \underline{Case 3 : $\alpha\in(2,3)$.} We first show that for all $t>0$, 
    \begin{equation}\label{eq:p2t_VS_pt}
        p_{2t}^{\Lambda_N}(0,0) \leq e^{-\lambda^{(1)}_{N, \alpha}t} p_t^{\Lambda_N}(0,0).
    \end{equation}
    Let $M=e^{L_{N, \alpha}}$ and abbreviate $\lambda:=\lambda^{(1)}_{N,\alpha}$; one has $\langle f,M^tf\rangle \leq e^{-\lambda t} \langle f,f\rangle$ for all $f \in \Omega^\R_N$, and we claim that this implies $\langle f,M^{2t}f\rangle \leq e^{-\lambda t} \langle f,M^tf\rangle$ (this is easy to check by decomposing $f$ on a basis of eigenfunctions of $M^t$). This gives \eqref{eq:p2t_VS_pt} when applied to $f=\delta_0$.
    Then, we have
    \begin{align*}
        G_{N, \alpha}(0,0) &=2 \int_0^{+\infty} p_{2t}^{\Lambda_N}(0,0) \mathrm{d}t \\
        &\leq 2\int_0^{+\infty} e^{-\lambda t} p_t^{\Lambda_N}(0,0) \mathrm{d}t \\
        &\leq C\int_0^{+\infty} \dfrac{e^{-\lambda t}}{t^b} \mathrm dt \qquad \text{ with }b=1/(\alpha-1),
    \end{align*}    
    where the last inequality relies on \eqref{eq:bound_pt(0,x)_alpha<3}. The integral on the RHS is finite since $b<1$ as $\alpha>2$.
    Thanks to the change of variable $u=\lambda t$, we get
    $\displaystyle\int_0^{+\infty} \dfrac{e^{-\lambda t}}{t^b} \mathrm dt=\lambda^{b-1}\int_0^{+\infty} \dfrac{e^{-u}}{u^b} \mathrm dt = C \lambda^{b-1}$. Combining with the previous equation and with \eqref{eq:smallest_eig_val}, we obtain
    $$G_{N, \alpha}(0,0) \leq C \lambda^{b-1} \leq CN^{\alpha-2}.$$

    \underline{Case 4 : $\alpha=3$.} We adapt the proof of case 3 but using \eqref{eq:bound_pt(0,x)_alpha=3} to upper bound $p_t^{\Lambda_N}(0,0)$. As this would lead to an upper bound of order $N/\sqrt{\ln N}$ instead of $N/\ln N$, we also split the integral into two parts:
    \begin{align*}
        G_{N,3}(0,0) &\leq C \int_0^{+\infty} \dfrac{e^{-\lambda t}}{\sqrt{1+t \ln |t|_+}} \mathrm dt \\
        &\leq C\int_0^N \dfrac{e^{-\lambda t}}{\sqrt{t \ln 2}} \mathrm dt+C\int_N^{+\infty} \dfrac{e^{-\lambda t}}{\sqrt{t \ln N}} \mathrm dt \\
        &\leq C \int_0^N \frac{1}{\sqrt t} \mathrm dt + \dfrac{C}{\sqrt{\ln N}} \lambda^{-1/2} \int_{\lambda N}^{+\infty} \frac{e^{-u}}{\sqrt u} \mathrm du \\
        &\leq C \sqrt{N}+C \dfrac{N}{\ln N} \leq C \dfrac{N}{\ln N},
    \end{align*}
    by \eqref{eq:smallest_eig_val}, which implies $\lambda^{-1/2} \leq C N/\sqrt{\ln N}.$

    \smallskip
    \underline{Case 5 : $\alpha=2$.}
    For $t\in [0,N]$, we directly apply \eqref{eq:bound_pt(0,x)_alpha<3} to bound $p_{t}^{\Lambda_N}(0,x)$. However, for the contribution of $(N, +\infty)$ to the integral, we will use the inequality $$p_{2t}^{\Lambda_N}(0,x) \leq p_{2t}^{\Lambda_N}(0,0) \leq C \dfrac{e^{-\lambda t}}{t},$$ which holds by monotonicity of $p_{2t}^{\Lambda_N}(0, \cdot)$, \eqref{eq:p2t_VS_pt} and \eqref{eq:bound_pt(0,x)_alpha<3}.
    It gives
    \begin{align*}
        G_{N,2}(0,x) &= \int_0^N p_t^{\Lambda_N}(0,x) \mathrm dt +2\int_{N/2}^{+\infty}p_{2t}^{\Lambda_N}(0,x) \mathrm dt \\
        &\leq \int_0^{|x|_+} \dfrac{1+t}{|x|_+^2} \mathrm dt +\int_{|x|_+}^N \frac1t \mathrm dt+ C \int_{N/2}^{+\infty} \dfrac{e^{-\lambda t}}{t} \mathrm dt \\
        &\leq C + 2\ln(N/|x|_+) + C \int_{\lambda N/2}^{+\infty} \frac{e^{-u}}{u} \mathrm{d}u \\
        &\leq C\left(1+\ln(N/|x|_+) \right),
    \end{align*}
    where in the last inequality we use the fact, coming from \eqref{eq:smallest_eig_val}, that $\lambda N \geq c$ (so that the last integral is finite).
\end{proof}

\printbibliography
\end{document}